\documentclass[12pt, reqno]{amsart}
\usepackage{amsmath,amssymb}
\usepackage{graphicx}
\usepackage{mathptmx}      
\usepackage{latexsym}
\usepackage{xcolor}
\usepackage{amsthm}

\newtheorem{proposition}{Proposition}[section]
\newtheorem{theorem}[proposition]{Theorem}
\newtheorem{definition}[proposition]{Definition}
\newtheorem{corollary}[proposition]{Corollary}
\newtheorem{lemma}[proposition]{Lemma}
\newtheorem{remark}[proposition]{Remark}
\newtheorem{ej}{Example}
\definecolor{verde}{rgb}{0.0, 0.5, 0.0}
\definecolor{auburn}{rgb}{1, 0.5, 0.15}
\definecolor{bordo}{rgb}{.65, 0.05, 0.25}

\newcommand{\Iup}{\overline{I}}
\newcommand{\He}{\mathcal{H}}

\newcommand{\Se}{\mathcal{S}}
\newcommand{\Ee}{\mathcal{E}}

\begin{document}

\title{Superhedging Supermartingales}



\author{C. Bender}
\address[C. Bender]{Department of Mathematics\\
Saarland University\\
Campus E 2 4,
66123 Saarbr\"ucken
Germany}
\email{bender@math.uni-sb.de}
\author{S.E.Ferrando}
\address[S.E.Ferrando]{Department of Mathematics\\
Toronto Metropolitan University\\
350 Victoria Street}
\email{ferrando@torontomu.ca}

\author{K.Gajewski}
\address[K. Gajewski]{Department of Mathematics\\
Toronto Metropolitan University\\
350 Victoria Street}
\email{konrad.gajewski@ryerson.ca}

\author{A.L. Gonzalez}
\address[A.L. Gonzalez]{Department of Mathematics\\
Mar del Plata National University\\
Mar del Plata, Funes 3350, 7600 Argentina}
\email{algonzal@mdp.edu.ar}


\date{Received: date / Accepted: date}

\maketitle
\begin{abstract}
Supermartingales are here defined on a non-probabilistic setting  and can be interpreted solely in terms of superhedging operations.
The classical expectation operator is replaced
by a  pair of 
subadditive operators one of them providing a class of null sets and the other one acting as an outer integral.
These operators are motivated by a financial theory of no-arbitrage pricing. Such a setting extends
the classical stochastic framework by replacing the path space of the process by a trajectory set, while also  providing a financial/gambling interpretation based on the notion of superhedging.  The paper proves analogues of the following classical results:  Doob's supermartingale decomposition and Doob's pointwise convergence theorem for nonnegative supermartingales. The approach shows how linearity of the expectation operator can
be circumvented and how integrability properties in the proposed setting
lead to the special case of (hedging) martingales while no integrability
conditions are required for the general supermartingale case.


\end{abstract}


\section{Introduction}

The paper introduces a class of non-probabilistic supermartingales in a setting where a set of price scenarios (also called trajectories) is given along with the possibility to trade as price trajectories unfold over time.
Trajectories are sequences $S= (S_i)_{i \geq 0} \in \mathcal{S}$ with a common
origin $S_0= s_0$, where the set $\mathcal{S}$ substitutes the abstract sample space $\Omega$ of the probabilistic setting. Following ideas of the theory of {\it non-lattice} integration developed by Leinert  \cite{leinert} and K\"onig \cite{konig}, one can construct an outer integral operator, denoted by $\overline \sigma$, which corresponds to the superhedging price when trading takes place by means of some idealized class of linear combinations of buy-and-hold strategies, see \cite{ferrando}. The idealization facilitates to establish an analogue of Daniell's continuity from below condition  for this outer superhedging integral operator, which is a standing assumption for the main results of this paper. Considering an investor, who enters the market at any later time, we can naturally introduce a conditional version of this superhedging outer  integral operator, denoted by $\overline \sigma_j$, which gives rise to the notion of a superhedging supermartingale via the relation $\overline \sigma_j f_{j+1}\leq f_j$. More precisely, the latter relation is only required to hold outside a null set -- and it is an important subtlety of Leinert's integration theory, that the null sets are determined by a countably sub-additive operator, $\overline I$ say, which is closely related but, in general, different from the  outer integral $\overline \sigma$. 

Proofs of classical results (i.e. in an stochastic setting) that involve  supermartingales
rely, at one point or another, on properties of
conditional expectation operators as well as on some measure theory. In our non-probabilistic setting, also referred to as {\it trajectorial setting}, it turns out that the space of integrable functions (restricted on which the  superhedging outer integral operator acts linearly) is inconveniently small, see \cite{bender1} for details. This fact obstructs the naive strategy of emulating classical proofs by replacing the expectation operator with the non-classical superhedging integral, but suggests to work with the  superhedging outer integral instead.This is a sublinear operator with unrestricted domain
but its definition allows to bypass the need for the linearity of the expectation as well as the non availability of some classical limit theorems.

In this paper, we prove analogues of several classical results (see for example \cite{neveu} and \cite{follmer})
for supermartingales in the trajectorial setting. In particular, we derive a representation theorem for superhedging supermartingales. Our representation is of a similar type as the uniform Doob decomposition in discrete time (Theorem 7.5 in \cite{follmer}) or the optional decomposition in continuous time, see \cite{elkaroui, kramkov} in the classical setting or \cite{nutz2, nutz} for  non-dominated versions. As illustrated by an example, our  Doob decomposition can also be applied to trajectory sets which do not have any martingale measure and, thus, cannot be recovered   
by classical robust supermartingale decompositions.

 Combining the supermartingale representation theorem with a convergence result for martingale transforms in the trajectorial setting (derived in \cite{ferrando}), we can, moreover, prove an analogue 
of Doob's a.e. pointwise convergence theorem for non-negative supermartingales. 

As another application of the supermartingale representation theorem, we clarify the role of the two superhedging operators $\overline \sigma$ and $\overline I$. Theorem \ref{thm:sigma_is_correct} shows that the  superhedging outer integral $\overline \sigma$ indeed provides the `correct' superhedging price in the sense that for payoffs of finite maturity it coincides with the minimal superhedging cost within the class of linear combinations of buy-and-hold strategies up to the null sets induced by the $\overline I$ operator. 

Our work in the Leinert-K\"onig setting provides an independent meaning, purely financially motivated,
to the results listed above. An inspection of our proof techniques
shows also the need to rely on new and independent proof arguments.

\subsection{Relation to the Literature}
The paper could be loosely considered as being part of the literature on robust financial mathematics that weakens a-priori probabilistic modeling
hypothesis, or dispenses with them altogether. This literature ranges from
discrete-time model-free superhedging dualities (e.g., \cite{burzoni, burzoni2, burzoni3}) to extending stochastic calculus beyond its original settings (e.g., \cite{vovk2,perkowski, beiglbock}).

Our setting is, however, more closely related to the game theoretic approach to probability initiated by Shafer, Vovk and coauthors (see e.g. \cite{shafer, vovk1} and the references therein). On a technical level, a key difference between the Shafer/Vovk approach and our setting is that their conditional global upper expectation operator $\mathbb{E}_s$ satisfies the axiomatic properties of an outer expectation  in every situation (Proposition 8.3 in \cite{vovk1}), while our conditional outer integral operator $\overline \sigma_j$ may assign value $-\infty$ to any bounded function  on a null set, on which the conditional version of the continuity from below property fails. In view of Theorem \ref{thm:L} below such a failure of the conditional continuity from below property can have two origins: a) The trajectory set may run in a situation (or, node, as we call it), in which the stock price will move upwards for sure (or will move downwards for sure) leading to an obvious arbitrage opportunity; b) A sure arbitrage situation arises at some node by trading up to an unbounded investment horizon, when the trajectory set turns out to be trajectorially incomplete (in the sense of Section \ref{sec:(L)} below). To the best of our knowledge, the aforementioned types of arbitrage situations are not presently accommodated into the abstract game-theoretic setting (as presented in Chapter 7 of \cite{vovk1}), but they are detected as null sets by our $\bar I$-operator (as should be), see also Remark \ref{rem:ShaferVovk} below. Dealing with these additional null sets does not only lead to significant technical difficulties, but, in view of Theorem \ref{thm:K}, we are required to work with the two different families of conditional superhedging operators $\overline \sigma_j$ and $\overline I_j$ (as opposed to the single family of global conditional upper expectations $\mathbb{E}_s$ in the game-theoretic approach). In this sense, our paper builds a bridge between the game-theoretic approach of Shafer and Vovk and the theory of non-lattice integration developed by Leinert and K\"onig.

\subsection{Structure of the Paper}
The paper is organized as follows; Section \ref{definitions} provides the definitions of the basic superhedging operators, the crucial continuity from below property of the superhedging outer integral (see Definition \ref{lPropertyWithLimInf}) and the notion of conditionally integrable functions
(see Definition \ref{integrable}). Section \ref{towerProperty} defines supermartingales and  stopping times and provides some examples.  Section \ref{representationTheorem} proves our supermartingale representation theorem. Doob's pointwise convergence result for non-negative supermartingales is derived in Section \ref{sec:doobConvergence}. The relation of the two families of superhedging operators $\overline \sigma_j$ and $\overline I_j$ is discussed in Section \ref{sec:I_vs_sigma}. In particular, we show that these operators actually differ, if the conditional continuity from below property is only asked to hold almost everywhere, see Theorem \ref{thm:K} for the precise statement.
 Section
\ref{sec:(L)} is dedicated to establish the validity of the conditional continuity of below property at almost every node (see Theorem \ref{thm:L})  -- a key 
hypothesis in several of our results. Section \ref{sec:(L)} also provides sufficient condition for a related hypothesis used in Section \ref{sec:doobConvergence}. Additional examples are also introduced to illustrate our results. Technical developments and some proofs are collected in  Appendices \ref{Property L} and \ref{Paper 1}.
Appendix \ref{app:L} contains the proof of Theorem \ref{thm:L} and provides a more general result, namely Theorem \ref{proofOfL}, establishing continuity from below at a given node.

\section{Basic setting and Fundamental operators} \label{definitions}

\subsection{Trajectorial Setting} \label{trajectorialSetting}
\begin{definition}[{\bf Trajectory set}] {\cite[Definition 1]{ferrando}}
Given a real number $s_0$, a \emph{trajectory set}, denoted by $\mathcal{S}= \mathcal{S}(s_0)$, is a subset of
\[
\mathcal{S}_{\infty}(s_0) =\{S=(S_i)_{i\ge 0}: S_i\in \mathbb{R},~ S_0=s_0\}.
\]
We make fundamental use of the following \emph{conditional spaces}; for $S \in \Se$ and $j \geq 0$  set:
\begin{equation} \nonumber
\Se_{(S,j)}\equiv\{\tilde{S} \in \Se: \tilde{S}_i= S_i, ~~ 0 \le i \le j\},
\end{equation}
the notation $(S,j)$, henceforth referred as a {\it node},  will be used as a shorthand for $\Se_{(S,j)}$. 
\end{definition}

Notice $\Se_{(S,0)} = \Se$ and, if $\tilde{S}\in \Se_{(S,j)}$, then $\Se_{(\tilde{S},j)}=\Se_{(S,j)}$. Moreover for $j \le k$ it follows that $\Se_{(S,k)}\subset \Se_{(S,j)}$. On the other hand, for any fixed pair $j<k$, $\Se_{(S,j)}$ is the disjoint union of $\Se_{(\tilde{S},k)}$ with $\tilde{S}\in \Se_{(S,j)}$, since if $\hat{S}\notin\Se_{(\tilde{S},k)}$ then $\Se_{(\tilde{S},k)}$ and $\Se_{(\hat{S},k)}$ are disjoint.

{\it Local} properties are relative to a given node. The classification of distinct nodes is presented in the following definition:

\begin{definition}[{\bf Types of nodes}] \label{typesOfNodes} 
	Given a trajectory space $\Se$ and a node $(S,j)$:

	\begin{itemize}
		\item $(S,j)$ is called an \emph{up-down node} if
		\begin{equation} \label{upDownProperty}
			\sup_{\tilde{S} \in \Se_{(S, j)}}~~ (\tilde{S}_{j+1} - S_{j}) >    0\quad \mbox{and}\quad
			\inf_{\tilde{S} \in \Se_{(S, j)}} ~~(\tilde{S}_{j+1} - S_{j}) <    0.
		\end{equation}

		\item $(S,j)$ is called a \emph{flat node} if
		\begin{equation} \label{flat}
			\sup_{\tilde{S} \in \Se_{(S, j)}}~~ (\tilde{S}_{j+1} - S_{j}) =    0 =
			\inf_{\tilde{S} \in \Se_{(S, j)}} ~~(\tilde{S}_{j+1} - S_{j}).
		\end{equation}
	\end{itemize}
	\noindent $(S, j)$ is called an arbitrage-free node if (\ref{upDownProperty}) or (\ref{flat}) hold, otherwise it
	is called an \emph{arbitrage node}. An arbitrage node $(S, j)$ is said to be of \emph{type I}, if there exists $\hat{S}\in\Se_{(S, j)}$
	such that $\hat{S}_{j+1} = S_j$; otherwise it is said to be of \emph{type II}.
\end{definition}

In practice, the coordinates $S_i$ are multidimensional in order to allow for multiple sources of uncertainty. For simplicity we restrict to $S_i \in \mathbb{R}$, one can also extend the framework to allow for several coordinates $S_i^k$ (see \cite{degano2} and \cite{ferrando19}).

\vspace{.1in}
Besides the set $\mathcal{S}$, the other basic component are the \emph{portfolios} defined as follows.
\begin{definition}[{\bf Conditional portfolio set}] For any fixed $S\in \Se$ and $j\ge 0$, $\He_{(S,j)}$ will be the set of all sequences of functions $H = (H_i)_{i \geq j}$, where $H_i: \Se_{(S,j)} \rightarrow \mathbb{R}$ are non-anticipative in the sense: for all $\tilde{S},\hat{S}\in\Se_{(S,j)}$ such that $\tilde{S}_k = \hat{S}_k$ for $j \le k\le i$, then $H_i(\tilde{S}) = H_i(\hat{S})$ (i.e. $H_i(\tilde{S}) = H_i(\tilde{S}_0,\ldots, \tilde{S}_i)$). 
	
Again, we introduce the shorthand notation $\mathcal{H}= \He_{(S,0)}$.
\end{definition}

Observe that since there are not restrictions to the conditional portfolio sets (i.e. portfolios acting on $\mathcal{S}_{(S,j)}$), then $\mathcal{H}_{(\tilde{S},j)}=\mathcal{H}_{(S,j)}$ whenever $\tilde{S}_i = S_i$, $0 \leq i \leq j$. Alternatively, one can
start with the `global' portfolios in $\mathcal{H}$ and define the sets $\mathcal{H}_{(S,j)}$ by restrictions to
 $\mathcal{S}_{(S,j)}$. Therefore, $H \in \He_{(S,j)}$ may be referred to as a \emph{ conditional portfolio}.
 
 \begin{remark}
 	Note that we do not impose any measurability condition on the functions $H_i:\mathbb{R}^{i+1}\rightarrow \mathbb{R}$ in the representation $H_i(\tilde S)=H_i(\tilde S_0,\ldots \tilde S_i)$ of a portfolio position. The main reason is that we will  work with a subadditive outer integral operator instead of a linear integral operator. On the one hand, this can be viewed in analogy to the use of the outer expectation operator in probability and statistics (see, e.g., \cite{VdVW}), which does not require any measurability properties of the integrands. On the other hand, this is in line with the protocols used in the discrete time game-theoretic approach of Shafer and Vovk \cite{shafer}, where no measurability conditions are imposed on the functions announced by the skeptic.        
 \end{remark}

For a node $(S,j)$, $H\in\He_{(S,j)}$, $V\in \mathbb{R}$ and $n\ge j$ we define $\Pi_{j,n}^{V, H}: \mathcal{S}_{(S,j)} \rightarrow \mathbb{R}$, as:
\begin{equation}   \nonumber 
\Pi_{j,n}^{V, H}(\tilde{S}) \equiv V+\sum_{i=j}^{n-1}H_i(\tilde{S})~ \Delta_i \tilde{S},\quad
\mbox{where}\;\: \Delta_i \tilde{S} = \tilde{S}_{i+1}- \tilde{S}_i,~~i\ge j, ~~\tilde{S} \in \mathcal{S}_{(S,j)}.
\end{equation}

\vspace{.1in} Notice that $V$ is a constant on $\mathcal{S}_{(S,j)}$ and so its value could change with $S$ i.e. $V= V(S)$.  In the sequel, being $\mathcal{A}$ a set of real valued functions, $\mathcal{A}^+$ will denote the set of its non-negative elements.
\begin{definition} [{\bf  Elementary vector spaces}]
\vspace{.05in} For a fixed node $(S,j)$  set
\begin{equation} \nonumber 
\mathcal{E}_{(S,j)}= \{f = \Pi_{j, n_f}^{V, H}: H \in \He_{(S,j)},~~V \in \mathbb{R}~~~\mbox{and}~~~n_f \in \mathbb{N}\}.
\end{equation}

%

Observe that $\mathcal{E}_{(S,j)}$ is a real vector space. Its elements are called \emph{elementary functions}. Let also define
\begin{equation} \nonumber 
\mathcal{E}_j = \{f:\Se\rightarrow \mathbb{R}: f|_{\Se_{(S,j)}}\in\mathcal{E}_{(S,j)} \;\; \forall S\in\Se \}.
\end{equation}
\end{definition}

\subsection{Fundamental Operators and Almost Everywhere Notions}\label{a.e. section}
Let $Q$ denote the set of all functions from $\mathcal{S}$ to $[-\infty, \infty]$ and $P \subseteq Q$ denotes the set of non-negative functions. The following conventions are in effect: $0 ~\infty =0$, $\infty + (- \infty) = \infty$, $u - v \equiv u +(-v) ~\forall~u, v \in [-\infty, \infty]$, and $\inf \emptyset = \infty$
(unless indicated otherwise).

We say that $f\in Q$ has \emph{maturity} $n_f\in \mathbb{N}$, if $f(S)=f(\tilde S)$ for every $S\in \Se$ and $\tilde S\in \mathcal{S}_{(S,n_f)}$, i.e. if $f$ depends on $S$ only through the first $n_f+1$ coordinates $S_0,\ldots, S_{n_f}$. In this case, we sometimes write $f(S_0,\ldots, S_{n_f})$ in place of $f(S)$. If $f$ has maturity $n_f$ for some $n_f\in \mathbb{N}$, we will speak of a function $f$ with finite maturity.

We define next the operator \; $\overline{I}_j : P \rightarrow \mathcal{E}^+_j$, which is a conditional extension of the operator $\overline I$ defined in \cite{ferrando} and it is used to define null sets. 

\begin{definition}\label{Iup_definition}
For a given node $(S, j)$ and a general $f \in P$ define
\begin{equation} \nonumber
\overline{I}_j f (S)\equiv  \inf \left\{\sum_{m \geq 1} V^m: ~~f \leq  \sum_{m \geq 1} \liminf_{n\to \infty}~~\Pi_{j, n}^{V^m, H^m}\;\;\;\mbox{on}\;\; \Se_{(S,j)},\; \Pi_{j, n}^{V^m, H^m}\in \mathcal{E}_{(S,j)}^+  ~~\forall ~~n \geq j \right\}.
\end{equation}
We will use the notation $\overline{I}f \equiv \overline{I}_0f$. We also set, for a general $f \in Q$:
\begin{equation} \nonumber
||f||_j(S) \equiv \overline{I}_j|f|(S)~~\mbox{and}~~||f|| \equiv ||f||_0(S).
\end{equation}
\end{definition}
Notice that $\overline{I}_j f (S)= \overline{I}_j f (S_0, \ldots, S_j)$, i.e. $\overline{I}_j f (\cdot)$ is constant on $\mathcal{S}_{(S,j)}$.
Moreover $\sum_{m \geq 1} V^m \geq 0$ hence, $\overline{I}_j f\ge 0$, so $||0||_j=0$.
$||\cdot||_j(S)$ will be called a {\it conditional norm}.

Next we introduce the notions of \emph{conditional null set} and the \emph{conditional $a.e.$ property}.
\begin{definition}[{\bf Conditional a.e. notions}] \label{nullObjects} Given a node $(S,j)$, a function $g\in Q$ is a \emph{conditionally null function at} $(S,j)$  if:\[\|g\|_j(S)=0.\]
 A subset $E\subset\Se$ is a \emph{conditionally null set at $(S,j)$} if $\|\mathbf{1}_E\|_j(S)=0$. A property is said to hold conditionally a.e. at $(S,j)$
(or equivalently: the property holds ``a.e. on $\mathcal{S}_{(S,j)}$") if the subset of  $\Se_{(S,j)}$ where it does not hold
is a conditionally null set at $(S,j)$. In particular, the latter definition applies to $g=f$ a.e. on $\mathcal{S}_{(S,j)}$, which also will be noted with $g\doteq f$ when $j=0$.

Notice that when $j=0$, the previous notions do not depend on $S$ and we apply the abbreviation ``a.e.'' for ``a.e. at $(S,0)$''. Moreover, $E\subset\Se$ is called a \emph{null set} and $g$ is called a \emph{null function}, if $\|\mathbf{1}_E\|=0$ and $\|g\|=0$, respectively.
\end{definition}
The next results, from \cite{ferrando}, gives properties of null functions and null sets that are widely used.
\begin{proposition} \label{propertiesOfIBarra} \cite[Proposition 1]{ferrando}
$\Iup$ is isotone,  positive homogeneous, countable subadditive and $\overline{I}(1_{\mathcal{S}}) \leq 1$.
\end{proposition}
\begin{proposition}\label{leinertTheorem} \cite[Proposition 2]{ferrando} Consider $f,g:\Se\rightarrow [-\infty,\infty]$, then
\begin{enumerate}
\item $\|g\| = 0$ iff $g = 0$ a.e.
\item The countable union of null sets is a null set.
\end{enumerate}
\end{proposition}

All appearing equalities and inequalities are valid for all points in the spaces where the functions are defined unless qualified by an explicit a.e. The notation
$f \dot{=} g$ is also used for the equality being valid only a.e.

\vspace{.1in}
We introduce next the operator  $\overline{\sigma}_j : Q \rightarrow \mathcal{E}_j$,
which we will call a \emph{conditional superhedging outer integral}
(or conditional outer integral); it is the main tool to define  the notion of \emph{trajectorial supermartingales}, the main object of study in our paper.

\begin{definition}[{\bf  Conditional Outer Integral}] \label{cond_integ_def} For a node $(S,j)$ and a general $f \in Q$,
\begin{equation} \nonumber
\overline{\sigma}_j f(S) \equiv  \inf \left\{\sum_{m \geq 0}V^m: ~~f \leq  \sum_{m \geq 0} f_m \;\;\mbox{on}\; \Se_{(S,j)}\right\},
\end{equation}
where
\[f_0=\Pi_{j, n_0}^{V^0,H^0}\!\!\in \mathcal{E}_{(S,j)}; \; \mbox{for}\; m \geq 1, \; f_m= \liminf_{n\to\infty}~~\Pi_{j, n}^{V^m, H^m},~~\mbox{and}\;\;\Pi_{j, n}^{V^m, H^m}\!\!\in \mathcal{E}_{(S,j)}^+~~\forall ~~~n \geq j.\]
Define also $\underline{\sigma}_j f(S) \equiv -\overline{\sigma}_j(-f) (S)$.
We will also set $\overline{\sigma}f \equiv \overline{\sigma}_0 f$.
\end{definition}

In some cases we may use the notation  $\overline{\sigma}_{(S,j)}f$
to make clear that the quantity $\overline{\sigma}_{(S,j)}f$ keeps
$(S,j)$ fixed. More common and useful is our reliance on the
defining notation $\overline{\sigma}_{j}f(S)$ treating $\overline{\sigma}_{j}f$
as a function on $\mathcal{S}$. Note that the initial endowments $V^m$ at node $(S,j)$  may depend on $S$ through $(S_0,\ldots,S_j)$ in all appearances.  

\if
The following definitions will help us to provide a more uniform presentation of the results.
We will set, for $S \in \mathcal{S}$,
\begin{equation} \label{atTheEndOfTime}
\mathcal{S}_{(S, \infty)} \equiv \{S\}~~~\mbox{and}~~\overline{\sigma}_{\infty} f (S)=  f(S)~\mbox{for}~ \in Q,
\end{equation}
we also set $\underline{\sigma}_{\infty} f(S)= - \overline{\sigma}_{\infty} (-f)(S)
= f(S)$.
\fi

\begin{remark}
\noindent Note that $\overline{\sigma}_j f(S)= \overline{\sigma}_j f(S_0, \ldots, S_j)$. Also, $f_0$ can also be written in a similar form as the $f_m$, $m \geq 1$ by mean of  $f_0=\liminf_{n\to\infty}~~\Pi_{j, n}^{V^0, H^0}$ with $H_i^0\equiv 0\,$ for $\,i\ge n_0$.
\end{remark}

Here we indicate the intuitions behind the definition of the operator
$\overline{\sigma}_jf$ (with analogous explanations for $\overline{I}_jf$). The main
simple portfolio superhedging $f$ is given by $f_0$, the role of the idealization
of an infinite number of  nonnegative  portfolios
$\sum_{m \geq 1} f_m$ is used to detect general type II arbitrage nodes 
as  null sets.
This is in close analogy to the use of elementary regions to define the area
of  non-elementary regions of the plane (for example). In our case,  elementary regions are replaced by simple portfolios, a class closed under linear combinations (but not being a vector lattice) and playing the role of the simple functions in Lebesgue's theory of integration. The quantity $\overline{\sigma}_jf(S)$
is then interpreted as a conditional outer integral (we are conditioning on $(S,j)$ and hence restricting
the future to $\mathcal{S}_{(S,j)}$), a functional version of
Carath$\acute{\mbox{e}}$odory's outer measure. Another idealization, which is used in model free finance (\cite{perkowski}, \cite{vovk2}),  is the presence of $\liminf$ in the functions $f_m$, $m \geq 1$, which should only play a meaningful role in situations where an arbitrary large number of portfolio rebalances is warranted. In our results such idealization is used to detect
further null sets that appear `at infinite time', for example, in Doob's pointwise convergence theorem.
In particular, financial positions are considered null if they are cheap to superhedge by means of idealized portfolios allowing an unbounded number of transactions.
We note that the mentioned idealizations are in line with the modern theory of integration
where countable number of operations are allowed. As an alternative to this approach one could study
the trajectorial analogue to the modern theory of integration such as the finitely additive theory developed in \cite{dubins}.

\if
 Having a stopping time $\tau$ (see Definition \ref{stoppingTimeDefinition})  with $\tau(S)$ finite, we can then write
$\overline{\sigma}_{(S, \tau(S))}f$ whenever we want to make clear that the conditioning
integer $j$ can change with $S$. Again, we will rely mostly in the function notation that, for $\tau(S)$ finite, becomes $\overline{\sigma}_{\tau}f$ which evaluates as follows $\overline{\sigma}_{\tau(S)}f(S)$ or more succinctly $\overline{\sigma}_{\tau}f(S)$.
\fi

\subsection{Property $(L)$}
\vspace{.1in}
The property $(L_{(S,j)})$, introduced in the following definition,
generalizes a non-conditional version from  \cite{leinert} and will be called (conditional) \emph{continuity from below}.
Property $(L_{(S,j)})$ will be key to obtain several of the results in the paper, in particular it will imply that $\overline{\sigma}_j$ preserves the property of non-negativity.

\begin{definition}[{\bf  Property $(L_{(S,j)})$}] \label{lPropertyWithLimInf}
For a fixed node $(S,j)$, $~~~~f = \Pi^{V, H}_{j, n_f}\in \mathcal{E}_{(S, j)}$
and $f_m = \liminf_{n \rightarrow \infty} \Pi^{V^m, H^m}_{j, n}$ with $\Pi^{V^m, H^m}_{j, n} \in \mathcal{E}_{(S, j)}^+~\mbox{for all}~ n \geq j~\mbox{and}~ m \geq 1$, ~define property $(L_{(S,j)})$ by
\begin{equation} \nonumber 
(L_{(S,j)}):\quad f \leq \sum_{m \geq 1} f_m\;\;\mbox{on}\; \Se_{(S,j)}
\implies V \leq \sum_{m \geq 1}  V^m.
\end{equation}
If $(L_{(S,j)})$ holds for $S\in\Se\;\; \overline{I}-a.e.$, it will be written $(L_j)$ holds $a.e.$ (see also Definition \ref{assumptionL-ae}).
\end{definition}
\begin{remark}
Property $(L_{(S,j)})$ is equivalent to
\begin{equation}  \nonumber
 0 \leq \sum_{m \geq 0} f_m\;\;\mbox{on}\; \Se_{(S,j)} \implies  0 \leq \sum_{m \geq 0} V^m.
\end{equation}
where the $f_m,~m \geq 1$, are as in Definition \ref{lPropertyWithLimInf} and
$f_0 \in \mathcal{E}_{(S,j)}$.
\end{remark}

Since $(L_{(S,0)})$ does not depend on $S$, it will be denoted by $(L)$. If $(L_{(S,j)})$ holds, it follows that $\overline{\sigma}_j f(S) = V$ for $f= \Pi^{V, H}_{j,n} \in \mathcal{E}_{(S, j)}$ (see Proposition \ref{Properties_L}). It will be proven in Proposition \ref{Properties_L} of Appendix \ref{Paper 1} that $(L_{(S,j)})$ is equivalent to several useful alternative statements. In particular,  property $(L_{(S,j)})$
is equivalent to $0 \leq \overline{\sigma}_j 0 (S)$ (and so to $0 = \overline{\sigma}_j 0(S)$) which in turn implies $\underline{\sigma}_j f(S)  \leq \overline{\sigma}_jf(S)$ for any $f:\mathcal{S} \rightarrow [- \infty, \infty]$ (i.e. $f \in Q$).
$(L_{(S,j)})$ will be established under appropriate conditions in Section \ref{sec:(L)} and Appendix \ref{app:L}.


\begin{definition}[{\bf  Conditional Integrable Functions}] \label{integrable}
$f \in Q$
is called a \emph{conditional integrable}  function at $j$ if it satisfies:
\begin{equation} \nonumber
\underline{\sigma}_j f = \overline{\sigma}_j f,~~ \overline{I}-a.e.
\end{equation}
In particular, $f \in Q$ is integrable if
 $\underline{\sigma}_0 f= \overline{\sigma}_0 f$.
\end{definition}
Integrable functions will play a side role in our work as we will make clear at due time.

Note that the conditional outer integral $\overline{\sigma}_j f(S)$ and the conditional inner integral  $\underline{\sigma}_j f(S)$ are defined at any node $(S,j)$ and for any function $f\in Q$.  However, it is possible to see that, $\overline{\sigma}_j f(S)=-\infty$ and  so $\underline{\sigma}_j f(S)=+\infty$ for any bounded function $f \in Q$, if  $(L_{(S,j)})$ fails at node $(S,j)$. Therefore, the assumption in Definition \ref{assumptionL-ae} below will be explicitly required for several results.

\vspace{.2in}
\begin{definition}[{\bf Assumption $(L)-a.e.$}] \label{assumptionL-ae}
The following two properties will be referred as {\it the assumption}
$(L)$-a.e.:
\begin{itemize}
\item[i)]  $(L)$ (i.e. $(L_{(S,0)}$)) holds, 

\vspace{.1in} 
\item[ii)]  \begin{equation}\label{eq:N^L}
 	\mathcal{N}^{(L)} \equiv \{S\in\Se: \exists j\ge 0 \; s.t. \; (L_{(S,j)})\; \mbox{fails}\}
 \end{equation}
 is a null set (in particular, $(L_j)$ holds a.e. for every $j$).
\end{itemize}
\end{definition}
 Sufficient conditions for this assumption and examples will be presented in Section~\ref{sec:(L)} and Appendix \ref{app:L}.

\begin{remark} \label{L0ForNonTrivialityOfComplement}
a) If (L) fails, then the set $\mathcal{N}^{(L)}$ in \eqref{eq:N^L} equals the whole trajectory set $\mathcal{S}$. Therefore, in order to avoid trivialities, we include property (L) in the definition of the assumption $(L)$-a.e.
\\[0.1cm] 
b)	The continuity property $(L)$ implies that $\overline{I}({\bf 1}_{\mathcal{S}})=1$ and so $1= \overline{I}({\bf 1}_{\mathcal{S}}) \leq
\overline{I}({\bf 1}_{\mathcal{S} \setminus \mathcal{M}})+ \overline{I}({\bf 1}_{\mathcal{M}}) \leq  \overline{I}({\bf 1}_{\mathcal{S} \setminus \mathcal{M}}) \leq 1$ where $\mathcal{M}$ is any arbitrary null set. In particular, $\overline{I}({\bf 1}_{\mathcal{S} \setminus \mathcal{N}^{(L)}})=1$ whenever $(L)$-a.e. holds.
\end{remark}

\begin{remark}\label{rem:ShaferVovk}
	In this remark we relate our setting, at a conceptual level, to the abstract discrete-time game-theoretic framework which is detailed in Chapter 7 of \cite{shafer}. In their setting, our nodes $(S,j)$ correspond to \emph{situations} $s$ and our (conditional) outer integrals $\overline \sigma_{(S,j)}$ can be considered as analogues of their \emph{upper expected values} $\overline{\mathbb{E}}_s$ in situation $s$, although there are some (subtle, but important) differences in the exact definition of both families of operators. Lemma 7.6 in \cite{shafer} (applied to a constant function $f\equiv c$) implies that $\overline{\mathbb{E}}_s[c]=c$ for every constant $c$ and every situation $s$. The analogue of such a condition in our setting is to ask that $(L_{(S,j)})$ holds for every node $(S,j)$. Our assumption $(L)$-a.e. is considerably weaker and allows, in view of Theorem \ref{thm:L}, to deal with arbitrage opportunities which arise at arbitrage nodes of type II or by trading on an unbounded time horizon (the latter ones in the case that the model is trajectorially incomplete, see Section \ref{sec:(L)}). 
\end{remark}

\section{Supermartingales: Definition and Examples}  \label{towerProperty}

Next is our definition of \emph{trajectorial supermartingales, submartingales, and martingales} and of \emph{stopping times}. 

\begin{definition} \label{superMartingale}
Consider a sequence $(f_j)_{j \geq 0}$ of non-anticipative functions $f_j:\mathcal{S} \rightarrow [- \infty, \infty]$, $j\ge0$. We say,

\vspace{.1in}\noindent $(f_j)$ is a \emph{supermartingale} if
\begin{equation}\label{supermartingale}
\overline{\sigma}_j f_{j+1} \leq f_j~~a.e.\;\; 0 \leq j <\infty,
\end{equation}
$(f_j)$ is a \emph{submartingale} if
\begin{equation}\label{submartingale}
f_j \leq \underline{\sigma}_j f_{j+1}~~a.e.\;\;0 \leq j < \infty,
\end{equation}
$(f_j)$ is a \emph{martingale} if
\begin{equation}\nonumber 
\underline{\sigma}_j f_{j+1}= \overline{\sigma}_j f_{j+1}= f_j~~a.e.\;\; 0 \leq j < \infty.
\end{equation}
\end{definition}
\begin{remark} \label{rem:supermartingale} Notice that if $(f_j)$ is a martingale, then according to Definition \ref{integrable}, $f_{j+1}$ is conditionally  integrable at $j$ for any $j\ge 0$.
 Moreover, if $(f_j)$ is a sub- and a supermartingale,  then, under Assumption $(L)$-a.e., $(f_j)$ is a martingale, since  
 by Proposition \ref{Properties_L} in Appendix \ref{Paper 1}\; $\underline{\sigma}_j f \le \overline{\sigma}_j f$ holds $a.e$ for any $f \in Q$.
\end{remark}

\begin{definition}[{\bf Stopping Time}, as per  Definition 8 in \cite{ferrando}]  \label{stoppingTimeDefinition}
	Given a trajectory space $\Se$, a \emph{trajectory based stopping time} (or \emph{stopping time} for short) is a function $\tau:\Se\rightarrow \mathbb{N}\cup\{\infty\}$ such that:
	\begin{equation*}
		\text{for any}\quad S,S'\in\Se\quad\text{if}\quad S_k=S'_k\quad\text{for}\quad 0\leq k\leq \tau(S),\quad\text{then}\quad \tau(S)=\tau(S').
	\end{equation*}
\end{definition}

We next provide examples for the above definitions. 

\begin{ej}\label{ex:supermart}
a) Suppose that $(L)$-a.e. holds. If $V\in \mathbb{R}$ and $(H_j)_{j\geq 0}$ is a non-anticipative sequence, then 
$$
M_j(S) \equiv V+\sum_{i=0}^{j-1} H_i(S)\Delta_iS,\quad S\in \mathcal{S},\;j\geq 0,
$$ 
satisfies the martingale property $\overline{\sigma}_j M_{j+1} (S)= M_j(S)$ whenever $(L_{(S,j)})$ holds (this is so by Proposition \ref{Properties_L} item $(4)$). In particular, the projection maps $(T_j)_{j\geq 0}$, $T_j(S) \equiv S_j$ are a martingale sequence (with $H_i\equiv 1$ and $V=S_0$). 

\vspace{.1in}
 b) For any $f\in Q$, the sequence $(f_j)_{j \geq 0}$ defined by
$$
f_j(S) \equiv \overline{\sigma}_jf(S),\quad S\in \mathcal{S},\;j\geq 0,
$$
forms a supermartingale and the sequence $(f_j)$ defined by
$$
f_j(S) \equiv \underline{\sigma}_jf(S),\quad S\in \mathcal{S},\;j\geq 0,
$$
forms a submartingale by the tower property in Proposition \ref{mainDirectionOfTowerProperty s-t} below.

\vspace{.1in}
c) If $(f_j)_{j \geq 0}$ is a supermartingale and  $(D_j)_{j\geq 0}$ is a non-anticipative sequence of nonnegative functions, then the \emph{supermartingale transform} $(g_j)_{j\geq 0}$
$$
g_j(S) \equiv f_0 + \sum_{i=0}^{j-1} D_i (f_{i+1}- f_i),\quad S\in \mathcal{S},\;j\geq 0,
$$ 
is again a supermartingale. This follows from subadditivity of $\overline{\sigma}_j$ and the remark that:
$\overline{\sigma}_j g_{j+1}(S) \leq \overline{\sigma}_j g_{j}(S) + \overline{\sigma}_j(D_j (f_{j+1} - f_j))(S) \leq g_j(S)+ D_j(S)~\overline{\sigma}_j (f_{j+1}- f_j)(S) \leq g_j(S)+ D_j(S)[ \overline{\sigma}_j f_{j+1}(S) - f_j(S)] \leq
g_j(S)$ (where we relied on Proposition \ref{requiredProperties}).

If $(f_j)$ is a submartingale, then we call $(g_j)$ (defined as above) a submartingale transform, which by the duality   $f_j \rightarrow -f_j$, $\overline{\sigma}_j \rightarrow \underline{\sigma}_j$, is a submartingale. 

\vspace{.1in}
d) If $(f_j)_{j \geq 0}$ is a supermartingale and $\tau$ is a stopping time, then the \emph{stopped sequence} 
$(f_{ j}^{\tau})_{j \geq 0}$  defined by 
\begin{equation} \nonumber
	f_j^{\tau}(S) \equiv f_{\tau(S)\wedge j}(S),
\end{equation}
is a supermartingale. This is a consequence of the previous item with the choice 
$$
D_i(S)=\begin{cases} 1,& \tau(S)>i \\ 0, & \tau(S)\leq i \end{cases} ,\quad S\in \mathcal{S},\;j\geq 0,
$$
which is non-anticipative by Lemma \ref{st-portfolio} below.
\end{ej}


\begin{lemma}\label{st-portfolio} Let $\tau$ a stopping time and $H^k=(H^k_i)_{i\ge 0},\;k=1,2$ sequences of non-anticipative functions. For $S\in\Se, j\ge 0$ define the following functions on $\mathcal{S}_{(S,j)}$:
\[
H^{\tau}_i(\tilde{S})=\left\{
\begin{array}{lll}
H^1_i(\tilde{S}) & \mbox{if} & j\le i < \tau(\tilde{S})\\
H^2_i(\tilde{S}) & \mbox{if} & \tau(\tilde{S})\le i.
\end{array}
\right.
\]
Then $H^{\tau}=(H^{\tau}_i)_{i\ge 0}$ is a sequence of non-anticipative functions.
\end{lemma}
\begin{proof}
Let $\tilde{S}_k=S_k,\; j\le k\le i$. If $j\le \tau(S)\le i \Rightarrow \tau(\tilde{S})=\tau(S)\le i\; \& \; H^{\tau}_i(\tilde{S})=H^2_i(\tilde{S})=H^{\tau}_i(S)$. While, by symmetry in previous reasoning, $i<\tau(S) \Rightarrow i<\tau(\tilde{S})\; \& \; H^{\tau}_i(\tilde{S})=H^1_i(S)=H^1_i(\tilde{S})=H^{\tau}_i(\tilde{S})$.
\end{proof}

\vspace{.1in}
\begin{proposition} [{\bf  Tower Inequality}] \label{mainDirectionOfTowerProperty s-t}
Let  $S$ be an arbitrary element of $\mathcal{S}$ and  $j\leq k$ non-negative integers; also let $f \in Q$. Then, 
\begin{equation}  \label{tower_ineq_0 s-tGeneralVersion}
\overline{\sigma}_{j}( \overline{\sigma}_{k}~f)(S) \le \overline{\sigma}_{j} f(S).
\end{equation}

\end{proposition}
\begin{proof}
We begin with the following inequality valid on $\mathcal{S}_{(S,j)}$:\\ $f \leq \Pi_{j, n_0}^{V^0, H^0}+ \sum_{m \geq 1} \liminf_{n \rightarrow \infty} \Pi_{j, n}^{V^m, H^m}$ where $\Pi_{j, n_0}^{V^0, H^0}\in \mathcal{E}_{(S,j)}$  and  $\Pi_{j, n}^{V^M, H^m} \in \mathcal{E}^+_{(S,j)}$ for every $n\geq j$. This inequality implies that $\overline{\sigma}_k f(\tilde S) \leq
\sum_{m\geq 0} \Pi_{j, k}^{V^m, H^m}(\tilde S)$ holds for all $\tilde S \in \mathcal{S}_{(S,j)}$, which in turn implies $\overline{\sigma}_j(\overline{\sigma}_k f)(S) \leq \sum_{m \geq 0} V^m(S)$. Therefore the result follows by the definition of $\overline{\sigma}_j$ as an infimum.
\end{proof}

We next  present sufficient conditions for the tower inequality to turn into an equality (and so obtaining an analogous result to the classical tower property), which in turn implies that $(\overline{\sigma}_jf)_{j\geq 0}$ is a martingale sequence.

\begin{corollary}\label{cor:martingale}
	Assume $(L)$-a.e. and fix some  $f \in Q$. Then, $(\overline{\sigma}_j f)_{j\geq 0}$ is a  martingale whenever $a)$ or $b)$ below hold:
	\begin{itemize}
		\item[a)] 	$f$ is conditionally integrable for every $j\geq 0$;
		\item[b)]   $f$ is integrable and $\overline{I}=\overline{\sigma}$ on positive functions with finite maturity.
	\end{itemize}
	\end{corollary}
\begin{proof}
	We start with the following preliminary considerations: 
	 Since $(L)$-a.e. holds, we conclude from Proposition \ref{Properties_L} that $\underline{\sigma}_j g\leq \overline{\sigma}_j g$ a.e. for every $g\in Q$ and $j\geq 0$. Applying this twice with $g=f$ and with  $g=\overline{\sigma}_k f$, in view of Corollary \ref{cor:leq_ae}, the following chain of inequalities holds a.e.
	$$
	\underline{\sigma}_j[\underline{\sigma}_k f]\leq \underline{\sigma}_j[\overline{\sigma}_k f]\leq \overline{\sigma}_j[\overline{\sigma}_k f].
	$$
	Applying Theorem \ref{mainDirectionOfTowerProperty s-t}, we obtain,
	\begin{equation}  \label{tower_ineq s-t}
		\underline{\sigma}_{j} f\le \underline{\sigma}_{j}[\underline{\sigma}_{k} f]\le \underline{\sigma}_{j}[\overline{\sigma}_{k} f] \leq \overline{\sigma}_{j}[\overline{\sigma}_{k} f]\le \overline{\sigma}_{j} f,\quad a.e.
	\end{equation}

	a)	The conditional integrability assumption now turns all a.e.-inequalities in \eqref{tower_ineq s-t} into a.e.-identities. In particular,  $(\overline{\sigma}_j f)_{j\geq 0}$ is a  martingale.
	
	b) Let $j=0$. Then, the  integrability assumption turns all inequalities in \eqref{tower_ineq s-t} into identities. In particular,
	we obtain $\overline{\sigma}[\overline{\sigma}_kf]=\underline{\sigma}[\underline{\sigma}_kf]$,  $\overline{\sigma}[\overline{\sigma}_kf]=\underline{\sigma}[\overline{\sigma}_kf]$ and also $\underline{\sigma}[\underline{\sigma}_kf]=\overline{\sigma}[\underline{\sigma}_kf]$ (with $-f$ in place of $f$).
	We then have access to Corollary \ref{linearityOfSigma}  to compute
	\begin{equation} \nonumber
		\overline{\sigma}[\overline{\sigma}_k f - \underline{\sigma}_k f] =  \overline{\sigma}[\overline{\sigma}_k f]+ \overline{\sigma}[-\underline{\sigma}_k f]= \overline{\sigma}[\overline{\sigma}_k f]- \underline{\sigma}[\underline{\sigma}_k f]=0.
	\end{equation}
	Therefore, given that $\overline{\sigma}_k f - \underline{\sigma}_k f \geq 0$ a.e.  (as per Proposition \ref{Properties_L}) we have
	$$\overline{I}[(\overline{\sigma}_k f - \underline{\sigma}_k f)_+]=
	\overline{\sigma}[(\overline{\sigma}_k f - \underline{\sigma}_k f)_+]=\overline{\sigma}[\overline{\sigma}_k f - \underline{\sigma}_k f]=0,$$ which, by Proposition \ref{leinertTheorem} item (1), implies $\overline{\sigma}_k f - \underline{\sigma}_k f \leq 0$ a.e. The two inequalities together yield  $\overline{\sigma}_k f = \underline{\sigma}_k f$ a.e., hence the conditional integrability of $f$ at $k$. Since $k$ is arbitrary, b) is reduced to a). 
\end{proof}

\section{Supermartingale decomposition}\label{representationTheorem}

In this section we prove a supermartingale representation theorem. It can be considered as an analogue of the uniform Doob decomposition in discrete time (see, e.g., Theorem 7.5 in \cite{follmer}) or the optional decomposition theorem in continuous time \cite{kramkov}, which apply to stochastic processes that are supermartingales simultaneously under a family of probability measures. 

\begin{theorem}[{\bf Supermartingale decomposition}]\label{thm:decomposition} 
Under Assumption $(L)$-a.e.,	let $(f_j)_{j\ge 0}$ be a sequence of
	non-anticipating real-valued functions. Then, the following assertions are equivalent:
	\begin{itemize}
		\item [(i)] $(f_j)_{j\ge 0}$ is a supermartingale.
		\vspace{.05in}
		\item [(ii)] For every sequence $(\delta_j)_{j\ge 0 }$ of positive real numbers there are sequences $(H_j)_{j\ge 0}$ and $(A_j)_{j\ge 0}$, of non-anticipating real-valued functions defined on 
		$\mathcal{S}$, such that $(A_j)_{j\ge 0}$ is nondecreasing, $A_0 = 0$, and
		\[
		f_i(S)=f_0+\sum_{j=0}^{i-1}H_j(S)\Delta_jS-A_i(S)+\sum_{j=0}^{i-1}\delta_j,
		\]
		for every $S \in \Se \setminus N_f$ and $i\ge 0$. Here $N_f$ is an $\overline{I}$-null set independent of  $(\delta_j)_{j\ge 0 }$.
	\end{itemize}
\end{theorem}

\begin{remark}
	Theorem \ref{thm:decomposition} shows that up to the small $\delta$-errors and null sets, supermartingales can be decomposed into a difference of martingale (of the special form as in Example \ref{ex:supermart}-(a)) and a non-anticipative, nondecreasing sequence.  	
	
	We will illustrate the supermartingale decomposition theorem, its assumptions, and its applicability beyond the classical probabilistic setting in a series of examples in Section \ref{sec:(L)}.  
\end{remark}


The proof of Theorem \ref{thm:decomposition} relies on two lemmas, which we call Finite Maturity Lemma and Aggregation Lemma.

\begin{lemma}[{\bf  Finite Maturity}] \label{finiteMaturitySuperhedging}
	Suppose $f: \mathcal{S} \rightarrow \mathbb{R}$ has maturity $n_f$ for some $n_f\in \mathbb{N}$. Let $j \leq n_f$ and $S^{\ast} \in \mathcal{S}$ be such that the property $(L_{(S^{\ast}, n_f)})$ holds.
	
	If $f_m=\liminf_{n\to\infty}~\Pi_{j,n}^{V^m,H^m},~~\Pi_{j,n}^{V^m,H^m} \in \mathcal{E}^+_{(S^{\ast}, j)}\; m\ge 1,~\mbox{and}~~
	f_0=~\Pi_{j,n_0}^{V^0,H^0} \in \mathcal{E}_{(S^{\ast}, j)}$  satisfy
	\begin{equation} \label{earlierCovering}
		f\le \sum\limits_{m=0}^\infty f_m,\;\;~~\mbox{on}~~\mathcal{S}_{(S^{\ast},j)}.~~
	\end{equation}
	Then
	\begin{equation}\label{toProve}
		f \le \sum_{m=0}^\infty \Pi_{j,n_f}^{V^m,H^m}\;\;\mbox{on}~~ \mathcal{S}_{(S^{\ast}, n_f)}~~ (\mbox{each side is constant on} ~~ \mathcal{S}_{(S^{\ast}, n_f)}).
	\end{equation}
\end{lemma}
\begin{proof}
	Define for each $\tilde{S} \in \mathcal{S}_{(S^{\ast}, n_f)}$,
	\begin{equation} \nonumber
		U^0(S^{\ast}) \equiv -f(S^{\ast})+\Pi_{j,n_f \wedge n_0}^{V^0,H^0}(S^{\ast}) = -f(S^{\ast})+\Pi_{j,n_f}^{V^0,H^0}(S^{\ast}), ~~
	\end{equation}
	\begin{equation} \nonumber
		g_0(\tilde{S}) \equiv U^0(S^{\ast})+ \sum_{i= n_f \wedge n_0}^{n_0-1} H_i^0(\tilde{S}) \Delta_i \tilde{S}= U^0(S^{\ast})+ \sum_{i= n_f}^{n_0-1} H_i^0(\tilde{S}) \Delta_i \tilde{S},
	\end{equation}
	where we have used the fact that $H_i^0=0$ for $i \geq n_0$,
	and for $m \geq 1$
	\begin{equation} \nonumber
		U^m(S^{\ast}) \equiv \Pi_{j,n_f}^{V^m,H^m}(S^{\ast}), \qquad
		g_m(\tilde{S}) \equiv U^m(S^{\ast})+ \liminf_{n \rightarrow \infty} \sum_{i= n_f }^{n-1} H_i^m(\tilde{S}) \Delta_i \tilde{S}.
	\end{equation}
	
	\noindent
	It follows that
	\[
	\Pi_{n_f,n}^{U^0(S^{\ast}),H^0}=-f(S^*)+\Pi_{j,n}^{V^0,H^0} \in \mathcal{E}_{(S^*,n_f)}~\mbox{for any}~  n \geq n_f,
	\]
	and for $m \geq 1$
	\[
	\Pi_{n_f,n}^{U^m(S^{\ast}),H^m}=\Pi_{j,n}^{V^m,H^m} \in \mathcal{E}^+_{(S^*,n_f)}~\mbox{for any}~  n \geq n_f,
	\]
	
	Notice that (\ref{earlierCovering}) implies
	$0 \leq \sum\limits_{m \geq 0} g_m$ holds on $\mathcal{S}_{(S^{\ast}, n_f)}$
	and since property $(L_{(S^*,n_f)})$ holds,
	\[
	0\le \sum\limits_{m=0}^\infty U^m(S^{\ast}) = -f(S^*)+\sum\limits_{m=0}^\infty\Pi_{j,n_f}^{V^m,H^m}(S^*),
	\]
	from where (\ref{toProve}) holds.
\end{proof}

The following aggregation lemma is proved in \cite{ferrando} (Lemma 3) under the assumption that all nodes are up-down nodes. The proof there can easily be adapted to our more general setting (which allows for any type of node).
\begin{lemma}[{\bf Aggregation Lemma}] \label{convergenceOfPortfolioCoordinates}
	For $j \ge 0$ fixed let, for $m\ge 1$  $~~H^m=(H_i^m)_{i\ge j}$, be sequences of non-anticipative functions on $\Se$, and $V^m$ functions defined on $\Se$, depending for each $S$ only on $S_0,...,S_j$.
	
	Fix a node $(S,j)$ and  assume for any $m\ge 1$, $n\geq j$, and  $\tilde{S}\in\Se_{(S,j)}$ that:
	\[
	\Pi_{j,n}^{V^m, H^m}(\tilde{S})=V^m(S)+\sum\limits_{i=j}^{n-1}H_i^m(\tilde{S})\Delta_i\tilde{S}\ge 0,~~\; \quad \mbox{and}\quad\sum\limits_{m\ge 1}V^m(S) <\infty.
	\]
	For every $\hat{S}\in \Se_{(S,j)}$ and $k\geq j$ the following holds: If  $(\hat{S}, k)$ is an up-down node and, for every $j\leq p<k$,  $(\hat{S}, p)$ is an up-down node or $\hat S_{p+1}=\hat S_p$, then,
	\begin{equation} \label{convergentCase}
		\sum_{m\ge 1} H_k^m(\hat{S})~\mbox{converges in $\mathbb{R}$}.
	\end{equation}
\end{lemma}

\begin{proof}[Proof of Theorem \ref{thm:decomposition}] 
	Let
	\begin{eqnarray*}
		\mathcal{N}^{(I)} &=& \{S\in\Se: \exists j\ge 0 \; s.t. \; (S,j)\; \mbox{is a type I arbitrage node and}\; S_{j+1}\neq S_j\},
	 \\ \mathcal{N}^{(II)} &=& \{S\in\Se: \exists j\ge 0 \; s.t. \; (S,j)\; \mbox{is a type II arbitrage node}\},
	\end{eqnarray*}
	and recall that $\mathcal{N}^{(L)}$, defined in \eqref{eq:N^L}, is a null set by assumption. Note that $\mathcal{N}^{(I)}$ and $ \mathcal{N}^{(II)}$ are null sets by Lemma \ref{nullityOfN_j} and that $ \mathcal{N}^{(II)}\subset\mathcal{N}^{(L)}$ by Lemma \ref{priceMinusInfinity}.
	\\[0.2cm]
	(i) $\Rightarrow$ (ii): By the supermartingale property of $(f_j)_{j\geq 0}$ we may fix  an $\overline{I}$-null set $\mathcal{N}_f$ such that 
	$\overline{\sigma}_j f_{j+1}(S) \leq f_j(S)$ for every $S \in \Se \setminus \mathcal{N}_f$ and $j\ge0$. Let $N_f \equiv \mathcal{N}^{(I)}\cup\mathcal{N}^{(L)}\cup\mathcal{N}_f$. We first introduce the stopping time
	\[
	\tau^{\#}(S)= \inf \{k\ge 0: L_{(S,k)}\; \mbox{fails, or}\; \overline{\sigma}_k f_{k+1}(S) > f_k(S),\; \mbox{or}\]
	\[[(S,k-1)\; \mbox{is a type I arbitrage node and}\; S_k\neq S_{k-1}]\},
	\]
	with the convention $\inf \emptyset = +\infty$. Note that $\tau^{\#}(S) = \infty$, if $S\notin N_f$.
	
	Now we fix some $j\ge 0$ and choose a family $\{S^{\lambda}\}_{\lambda\in \Lambda_j}$ for some index set $\Lambda_j$ so that $\{\Se_{(S^{\lambda},j)}:\lambda\in \Lambda_j\}$ is a partition of $\Se$ (see Definition \ref{indexSet} in appendix \ref{Property L}).
	
	We construct now a function $H_j:\Se \rightarrow \mathbb{R}$ constant on the nodes of the partition (and, thus, non-anticipating) in the following way.
	Consider an arbitrary but fixed node $(S^{\lambda},j)$ of the partition:
	
	If $\tau^{\#}(S^\lambda)\le j$, then $\Se_{(S^\lambda,j)}\subset N_f$ and we simply let $H_j(S)=0$ for any $S\in\Se_{(S^\lambda,j)}$.
	
	If $\tau^{\#}(S^\lambda)\ge j+1$ note that, in particular, $(L_{(S^{\lambda},j)})$ holds and $\overline{\sigma}_j f_{j+1}(S^\lambda)\le f_j(S^\lambda) \in \mathbb{R}$. Applying the defnition of $\overline{\sigma}_j$, we  find $g_m$'s such that
	$$
f_{j+1} \leq g_m \mbox{ on } \Se_{(S^{\lambda},j)},
	$$
	where $g_m=\liminf_{n\to\infty}~\Pi_{j,n}^{V^m,H^m},~~\Pi_{j,n}^{V^m,H^m} \in \mathcal{E}^+_{(S^{\lambda}, j)}\; m\ge 1,~\mbox{and}~~
	g_0=~\Pi_{j,n_0}^{V^0,H^0} \in \mathcal{E}_{(S^{\lambda}, j)}$  and
\begin{equation}\label{eq:0002}
		\sum_{m=0}^{\infty}V_m \le f_j(S^{\lambda})+\delta_j.
\end{equation}
Notice that $H^m_j$ is constant on $\Se_{(S^{\lambda},j)}$, and we write $h_m$ for this constant value. 
If $(S^{\lambda}, j)$ is an up-down node, then $\sum_{m=0}^{\infty}h_m$ converges in $\mathbb{R}$ by the Aggregation Lemma~\ref{convergenceOfPortfolioCoordinates}. We now define $H_j(S) \equiv H_j(S^{\lambda})$ for $S\in\Se_{(S^{\lambda},j)}$ in the following way:
\begin{itemize}
	\item If $(S^{\lambda}, j)$ is an up-down node, then $H_j(S^{\lambda})\equiv\sum_{m=0}^{\infty}h_m$,
	\vspace{.05in}
	\item Otherwise $H_j(S^{\lambda})\equiv 0$.
\end{itemize}
If $S\in (S^{\lambda}, j)$ satisfies $(L_{(S,j+1)})$, then, by the Finite Maturity Lemma \ref{finiteMaturitySuperhedging} with $n_{f_{j+1}}=j+1$,
	\begin{equation}\label{eq:0003}
	f_{j+1}(S) \le \sum_{m=0}^{\infty}V_m + h_m\Delta_jS.
	\end{equation}
{\it Case A:}
	 $S \in (S^{\lambda}, j)$ satisfies $(L_{(S,j+1)})$ and [$(S^{\lambda},j)$ is an up-down or $S_{j+1}=S_j$]. 
	 
	 Then, by \eqref{eq:0002} and \eqref{eq:0003},
	\begin{equation} \nonumber 
		\begin{array}{lll}
			f_j(S)+\delta_j &=& f_j(S^{\lambda})+\delta_j \ge \sum\limits_{m=0}^{\infty}V_m\\ \\
			&\ge & f_{j+1}(S) - \sum\limits_{m=0}^{\infty} h_m\Delta_jS = f_{j+1}(S) - H_j(S)\Delta_jS.
		\end{array}
	\end{equation}
{\it Case B:}	For  $S\in\Se_{(S^{\lambda},j)}$, $(L_{(S,j+1)})$ fails or [$(S^{\lambda},j)$ is not an up-down and $S_{j+1}\neq S_j$]. 
Then, $\tau^{\#}(S)= j+1$.

\noindent
\vspace{.1in}		
Consequently, for every $S\in\Se$ such that $\tau^{\#}(S)\ge j+2$, and, hence in particular for every $S \notin N_f$, we are in Case A and
	\begin{equation}\label{eq:increment}
		f_{j+1}(S) - f_j(S) \le \delta_j + H_j(S)\Delta_jS.
	\end{equation}

	We now define, for $S\in\Se\setminus N_f$
	\[
	\alpha_j(S)\equiv \delta_j + H_j(S)\Delta_jS -(f_{j+1}(S) - f_j(S))\ge 0;
	\]
	and set $\alpha_j(S)\equiv 0$ for for $S\in N_f$. Then for $S\in\Se\setminus N_f$
	\[
	f_{j+1}(S) - f_j(S) = \delta_j + H_j(S)\Delta_jS -\alpha_j(S).
	\]
	Summing this identity over $j$, yields, for every $i\ge 0$ and $S\in\Se\setminus N_f$,
	\[
	f_i(S) = f_0 + \sum\limits_{j=0}^{i-1}H_j(S)\Delta_jS -A_i(S)+ \sum\limits_{j=0}^{i-1}\delta_j,
	\]
	with $A_i(S)\equiv \sum\limits_{j=0}^{i-1}\alpha_j(S)$.
	
	\vspace{.2cm}
	$(ii) \Rightarrow (i)$: As $N_f$ is an $\bar I$-null set, we may deduce from
	Proposition \ref{requiredProperties}--e), that $\bar I_j {\bf 1}_{N_f}=0$ $\bar I$-a.e. for every $j\ge 0$. Hence, there is an $\bar I$-null set $\mathcal{N}_f$ such that $\bar I_j {\bf 1}_{N_f}(S)=0$ for every $S\in \mathcal{S}\setminus \mathcal{N}_f$ and $j\geq 0$. 
	
	We now fix some $j\geq 0$ and some positive integer $K$ and let $\delta_i=1/K$. Condition (ii) implies
	$$
	f_{j+1}(S)\leq f_j(S)+ 1/K +  H_j( S)(S_{j+1}-S_j)
	$$
	for every $S\in \mathcal{S}\setminus {N}_f$. We define
	$$
	g_{j+1}(S)=  f_j(S)+ 1/K +  H_j( S)(S_{j+1}-S_j)
	$$
	for every $S\in \mathcal{S}$. Then, $g_{j+1}\in\mathcal{E}_j$ and $\overline\sigma_jg_{j+1}(S)\leq f_j(S)+1/K$ for every $S\in \mathcal{S}$.
	Note that $(f_{j+1}-g_{j+1})_+$ can only be strictly positive on the $\bar I$-null set $N_f$. Consequently,
	\begin{eqnarray*}
		\overline\sigma_j f_{j+1}(S) &\leq& \overline\sigma_jg_{j+1}(S)+ \bar I_j((f_{j+1}-g_{j+1})_+)(S) \\ &\leq& f_j(S)+1/K + \sum_{m=1}^\infty \bar I_j( {\bf 1}_{{N}_f} )(S)= f_j(S)+1/K
	\end{eqnarray*}
	for every $S\in \mathcal{S}\setminus\mathcal{N}_f $. Passing to the limit $K\rightarrow \infty$, yields (i).
\end{proof}

\section{Doob's Pointwise Supermartingale Convergence} \label{sec:doobConvergence}

This section proves our version of Doob's pointwise convergence theorem for nonnegative supermartingales. 

\begin{theorem}[{\bf Supermartingale convergence}] \label{doobConvergence}
	Suppose $(L)$-a.e. holds. Let $(f_i)_{i \ge 0}$ be a supermartingale with values in $[0,\infty)$ and impose the following assumption on the trajectory set:
	\begin{enumerate}
				\item[$(H.1)$] Whenever $(S,j)$ is an up-down node such that $(L_{(S,j)})$ holds and $(L_{(S,j+1)})$ fails, then there are $S^1, S^0\in \Se_{(S,j)}$ such that $S^1_{j+1}\geq S_{j+1}\geq S^0_{j+1}$ and $(L_{(S^\iota,j+1)})$ holds for $\iota=0,1$. 
	\end{enumerate}
		 Then, there exist a null set $\mathcal{N}_{div}$ such that $	\lim\limits_{i\to\infty} f_i(S)$ exists in $\mathbb{R}$ for every $S\in \Se\setminus\mathcal{N}_{div}$. 
	\end{theorem}

The proof's strategy is to apply the supermartingale decomposition in Theorem~\ref{thm:decomposition} and to pass to the limit separately for the various terms. For the martingale part $\sum_{j=0}^{i-1}H_j(S)\Delta_jS$, we can make use of Theorem 2 in \cite{ferrando}. However, this result requires that  $\sum_{j=0}^{i-1}H_j(S)\Delta_jS$ is bounded from below by the same constant for every $S\in \mathcal{S}$, while Theorem \ref{thm:decomposition} in conjunction with the nonnegativity of $f$ only implies boundedness of
$\sum_{j=0}^{i-1}H_j(S)\Delta_jS$ from below for a.e.  $S\in \mathcal{S}$. In view of the following lemma, the required boundedness condition can be guaranteed under the additional assumption $(H.1)$. Useful sufficient conditions for $(H.1)$ to hold are presented in Section \ref{sec:(L)}, Corollary \ref{cor:L}.

\begin{lemma} \label{lem:positive_martingale}
Under the assumptions of Theorem \ref{doobConvergence}, fix a sequence $(\delta_j)_{j \geq 0}$ of summable positive reals and construct the supermartingale decomposition of $(f_j)$ as in the proof of the implication $(i) \Rightarrow (ii)$ in Theorem \ref{thm:decomposition}. Then, for every $S\in \mathcal{S}$ and $i\geq 0$,
$$
f_0+\sum_{j=0}^\infty \delta_j +\sum_{j=0}^{i-1}H_j(S)\Delta_jS\geq 0.
$$
\end{lemma}
\begin{proof}
We construct $(H_j)$ as in the proof of  the implication $(i) \Rightarrow (ii)$ in Theorem \ref{thm:decomposition}. 
With $\tau^\#$ as in the proof of Theorem \ref{thm:decomposition}, we note that:
If $j\leq \tau^\#(S)-2$ or if [$j= \tau^\#(S)-1$ and $(S,j)$ is an up-down node and $(L_{(S,j+1)})$ holds], we are in Case A of the proof of Theorem \ref{thm:decomposition} and, thus, by \eqref{eq:increment}, 
\begin{equation}\label{eq:0004}
f_{j+1}(S)-f_j(S)\leq \delta_j +  H_j( S)(S_{j+1}-S_j)
\end{equation}
Moreover, if $j\geq \tau^\#(S)$ or  [$j= \tau^\#(S)-1$ and $(S,j)$ is not an up-down node], then 
\begin{equation}\label{eq:0005}
H_j(S)=0.
\end{equation}
We now distinguish two cases:
\\[0.2cm] 
(i) $i<\tau^\#(S)$. 

Then $j\leq \tau^\#(S)-2$ for any $j<i$, and, consequently, by \eqref{eq:0004},
$$
f_0+\sum_{j=0}^\infty \delta_j +\sum_{j=0}^{i-1}H_j(S)\Delta_jS\geq f_0 +\sum_{j=i}^\infty \delta_j +\sum_{j=0}^{i-1} f_{j+1}(S)-f_j(S)\geq f_i(S)\geq 0.
$$
\ \\[0.1cm]
(ii) $i\geq \tau^\#(S)$. 

Then, by \eqref{eq:0005} and \eqref{eq:0004},
\begin{eqnarray*}
&&	f_0+\sum_{j=0}^\infty \delta_j +\sum_{j=0}^{i-1}H_j(S)\Delta_jS \\ &=& f_0 +\sum_{j=\tau^\#(S)}^\infty \delta_j +\sum_{j=0}^{\tau^\#(S)-2} \delta_j +H_j( S)(S_{j+1}-S_j) \\ &&+ \left(\delta_{\tau^\#(S)-1}+H_{\tau^\#(S)-1}( S)(S_{\tau^\#(S)}-S_{\tau^\#(S)-1})\right) \\
&\geq & f_{\tau^\#(S)-1}(S)+\left(\delta_{\tau^\#(S)-1}+H_{\tau^\#(S)-1}( S)(S_{\tau^\#(S)}-S_{\tau^\#(S)-1})\right).
\end{eqnarray*}
If $(S,\tau^\#(S)-1)$ is not an up-down-node, then, by \eqref{eq:0005},
$$
f_0+\sum_{j=0}^\infty \delta_j +\sum_{j=0}^{i-1}H_j(S)\Delta_jS\geq  f_{\tau^\#(S)-1}(S)\geq 0.
$$
If  $(S,\tau^\#(S)-1)$ is an up-down-node and $(L_{(S,\tau^\#(S))})$ holds, then, by \eqref{eq:0004},
$$
f_0+\sum_{j=0}^\infty \delta_j +\sum_{j=0}^{i-1}H_j(S)\Delta_jS\geq  f_{\tau^\#(S)}(S)\geq 0.
$$
If  $(S,\tau^\#(S)-1)$ is an up-down-node and $(L_{(S,\tau^\#(S))})$ fails, then, by Assumption $(H.1)$ with $j=\tau^\#(S)-1$, there are $S^1,S^0\in (S,\tau^\#(S)-1)$ such that 
 $L_{(S^\iota,\tau^\#(S))}$ holds for $\iota= 0, 1$ and $S^0_{\tau^\#(S)}\leq  S_{\tau^\#(S)}\leq S^1_{\tau^\#(S)}$. If $H_{\tau^\#(S)-1}(S)\geq 0$, then, by invoking \eqref{eq:0004} for $S^1$,
 \begin{eqnarray*}
&& f_0+\sum_{j=0}^\infty \delta_j +\sum_{j=0}^{i-1}H_j(S)\Delta_jS\\ &\geq& f_{\tau^\#(S)-1}(S^1)+\left(\delta_{\tau^\#(S)-1}+H_{\tau^\#(S)-1}( S^1)(S^1_{\tau^\#(S)}-S^1_{\tau^\#(S)-1})\right)\geq f_{\tau^\#(S)}(S^1)\geq 0.
\end{eqnarray*}
If $H_{\tau^\#(S)-1}(S)< 0$, then, the same argument with $S^0$ in place of $S^1$ yields
$$
 f_0+\sum_{j=0}^\infty \delta_j +\sum_{j=0}^{i-1}H_j(S)\Delta_jS\geq f_{\tau^\#(S)}(S^0)\geq 0.
$$
\end{proof}

\begin{proof}[{\bf Proof of Theorem \ref{doobConvergence}}] 
	Fix a sequence $(\delta_j)_{j\ge 0}$ of positive reals with $\sum_j\delta_j<\infty$. By Theorem \ref{thm:decomposition}  there are sequences $({H}_j)_{j\ge 0}$, $(A_j)_{j\ge 0}$ of non-anticipating real-valued functions such that $(A_j)_{ j \geq 0}$ is nondecreasing, $A_0 = 0$, and
\begin{equation}\label{SupermartingaleDecomposition}
 f_i(S)=f_0 + \sum\limits_{j=0}^{i-1} {H}_j(S)\Delta_jS - A_i(S) + \sum_{j=0}^{i-1}\delta_j,
\end{equation}
for every $S\in\Se\setminus {N}_f$ and $i\ge 0$, with $N_f$ a null set independent of $(\delta_j)_{j\ge 0}$.
Besides, with ${V}\equiv f_0+\sum\limits_{j=0}^{\infty}\delta_j$ it follows by Lemma \ref{lem:positive_martingale} that
\begin{equation} \label{neededPositivityProperty}
\Pi^{{V},{H}}_i(S)=f_0+\sum_{j=0}^{\infty}\delta_j+\sum_{j=0}^{i-1}{H}_j(S)\Delta_jS\ge 0,
\end{equation}
for every $i\ge 0$ and $S\in\Se$.

Having in mind that there are no portfolio restrictions on $\He$,  from \cite[Theorem 2]{ferrando}  it follows that there exists a null set $N_0$ such that $\lim\limits_{i\to\infty}\Pi^{V,{H}}_i(S)$ exists and is finite for any $S\in\Se\setminus N_0$. Consequently $(\Pi^{{V},{H}}_i(S))_{i\ge 0}$ is bounded for those $S$. Let $\mathcal{N}_{div} \equiv N_0 \cup N_f$ and restrict the following arguments to $S \in \mathcal{S} \setminus \mathcal{N}_{div}$. Therefore, from (\ref{SupermartingaleDecomposition}) and the nonnegativity of $f$, $(A_i(S))_{i\ge 0}$ results bounded and since it is nondecreasing, $\lim_{i\to\infty}A_i(S)$ also exists (and so it is finite). So $\lim\limits_{i\to\infty} f_i(S)$ exists in $\mathbb{R}$  from (\ref{SupermartingaleDecomposition}), for any $S \in \Se \setminus \mathcal{N}_{div}$.
\end{proof}

\section{On the Relation between the Two Superhedging Operators}\label{sec:I_vs_sigma}

As another consequence of the supermartingale decomposition, we show, in this section, that $\overline \sigma$ is the `correct' superhedging operator in the sense that, for bounded (from below) functions with finite maturity, it corresponds to the minimal superhedging price within the usual class of simple portfolios up to the null sets induced by $\bar I$. We also clarify the relation between the two (conditional) superhedging operators $\overline I_j$ and $\overline \sigma_j$.

\begin{theorem}\label{thm:sigma_is_correct}
	Suppose that $(L)$-a.e. holds and that $f\in Q$ has maturity $n_f\in \mathbb{N}$, is bounded from below and satsifies $\overline \sigma f<\infty$. Let $0\leq j<n_f$. Then:
	For every $\epsilon>0$, there are a null set $N_f$ and a non-anticipative sequence 
	$(H_i)_{i=j,\ldots, n_f-1}$ such that for every $S\in \Se\setminus N_f$
	$$
	f(S)\leq (\overline\sigma_j f(S)+\epsilon)+ \sum_{i=j}^{n_f-1} H_i(S)\Delta_iS.
	$$
	Conversely, if there are a $V\in Q$ with maturity $j$ and a null set $\tilde N_f$ such that for every $S\in\Se \setminus \tilde N_f$
	$$
	f(S)\leq V(S_0,\ldots, S_j)+ \sum_{i=j}^{n_f-1} H_i(S)\Delta_iS,
	$$
	then $\overline \sigma_jf\leq V$ a.e.
	
	In particular,
	\begin{eqnarray*}
		\overline \sigma f&=\quad\inf\{V\in \mathbb{R}|& \exists (H_j)_{j=0,\ldots, n_f-1} \textnormal{ non-anticipative such that } \\ && V+\sum_{j=0}^{n_f-1} H_j(S)\Delta_jS \geq f(S) \textnormal{ for a.e. }S\in \Se \}.
	\end{eqnarray*}
\end{theorem}

The proof combines Theorem \ref{thm:decomposition} with the following lemma, which deals with the issue that the supermartingale $(\overline\sigma_i f)_{i\geq 0}$ may take values $\pm \infty$.
\begin{lemma}\label{lem:sigma_is_correct}
	Suppose $(L)$-a.e. If $f\in Q$ is bounded from below by some $c\in \mathbb{R}$ and satisfies $\overline \sigma f<\infty$, then there is a supermartingale $(f_i)_{i\geq 0}$ with values in $[c;+\infty)$ such that $\overline \sigma_j f=f_j$ a.e. for every $j\geq 0$. 
\end{lemma}
\begin{proof}
	As in the proof of Theorem \ref{thm:decomposition}, we consider the stopping time
	\[
	\tau^{\#}(S)= \inf \{k\ge 0: L_{(S,k)}\; \mbox{fails, or}\; [(S,k-1)\; \mbox{is a type I arbitrage node and}\; S_k\neq S_{k-1}]\}.
	\]
	Define
	$$
	f_i(S)=\overline \sigma_if(S){\bf 1}_{\{i<\tau^\#(S)\}} + c{\bf 1}_{\{i\geq\tau^\#(S)\}},\quad i\geq 0,\;S\in \mathcal{S},
	$$
    and note that $(f_i)_{i\geq 0}$ is non-anticipative by Lemma  \ref{st-portfolio}. Recall that $\{S\in \mathcal{S}|\; \tau(S)<\infty\}$ is a null set by Lemma \ref{nullityOfN_j} and by assumption $(L)$-a.e. In view of Example \ref{ex:supermart}--b), we conclude that $(f_j)$ is a supermartingale and $\overline \sigma_j f=f_j$ a.e. for every $j\geq 0$.
	
	For the lower bound, note that by Proposition \ref{Properties_L},
	$$
	f_i(S)=\overline\sigma_if(S)\geq \overline\sigma_i(c)(S)=c,
	$$
	whenever $i<\tau^\#(S)$. We finally need to verify that $f_i(S)<\infty$ for every $i\geq 0$ and $S\in \mathcal{S}$. Since $\overline \sigma f< \infty$, we find 
	$\Pi_{0, n_0}^{V^0,H^0}\!\!\in \mathcal{E}_0$ and, for every  $m\in \mathbb{N}$, $\Pi_{0, \cdot}^{V^m, H^m}$  such that $\Pi_{0, n}^{V^m, H^m}\!\!\in \mathcal{E}^+_0$ for every $n\geq 0$, $\sum_{m=1}^\infty V^m<\infty$ and
	$$
	f(S)\leq  \sum_{m=0}^\infty \liminf_{n\rightarrow \infty} \Pi_{0, n}^{V^m, H^m}(S),\quad S\in \mathcal{S}.
	$$
	We now fix a node $(S,i)$. The previous inequality implies 
	$$
	f(\tilde S)\leq   \sum_{m=0}^\infty \liminf_{n\rightarrow \infty} \Pi_{i, n}^{\tilde V(S), \tilde H^m}(\tilde S),\quad \tilde S\in \mathcal{S}_{(S,i)}
	$$
where $\tilde V(S)=\Pi_{0, i}^{V^m, H^m}(S)$ and $\tilde H^m=  H^m_{|\mathcal{S}_{(S,i)}}$. Hence, $\overline \sigma_i f(S)\leq \sum_{m=0}^\infty \Pi_{0, i}^{V^m, H^m}(S)$. We still need to check that the right-hand side is finite, if $i<\tau^\#(S)$. In this case, we may
  (thanks to Lemma~\ref{priceMinusInfinity}) apply the Aggregation Lemma \ref{convergenceOfPortfolioCoordinates} to conclude that 
$$
\sum_{m=0}^\infty \Pi_{0, i}^{V^m, H^m}(S)=\sum_{m=0}^\infty V^m +\sum_{j=0}^{i-1} \sum_{m=0}^\infty (H^m_j(S)\Delta_jS)
$$
where each of the series converges in $\mathbb{R}$.
\end{proof}

\begin{proof}[{\bf Proof of Theorem \ref{thm:sigma_is_correct}}]
	 Fix some $\epsilon>0$ and choose a sequence $(\delta_i)$ of positive reals such that $\sum_{i=j}^{n_f-1} \delta_j\leq \epsilon$. We fix a supermartingale $(f_i)_{i\geq 0}$ as in Lemma \ref{lem:sigma_is_correct} and note that, thanks to Lemma~\ref{integProperty} and assumption $(L)$-a.e., $f_i=\overline\sigma_if=f$ a.e. for $i\geq n_f$.
	 Hence,
	  we may apply Theorem \ref{thm:decomposition} to $(f_i)$ in order to construct a non-anticipative sequence $(H_i)_{i=0,\ldots, n_f-1}$ such that
	$$
	f(S)=f_{n_f}(S)\leq \left( \overline \sigma_j f+\epsilon\right) +\sum_{i=j}^{n_f-1}H_i(S)\Delta_iS 
	$$
	for every $S\in \Se\setminus N_f$, where $N_f$ is a null set. 
	
	For the converse inequality, assume that there are $V\in Q$ with maturity $j$, $H=(H_i)_{i=j,\ldots, n_f-1}$ non-anticipative, and a null set $\tilde N_f$ such that for every $S\in \Se\setminus \tilde N_f$
	$$
	f(S)\leq V(S_0,\ldots,S_j)+\sum_{i=j}^{n_f-1} H_i(S)\Delta_iS.
	$$
	If $V(S_0,\ldots,S_j)=+\infty$, the inequality $\overline \sigma_j f(S)\leq V(S)$ is trivial. Otherwise, we note that
	 $g=\infty {\bf 1}_{\tilde N_f}$ is a null function. Hence, by 	Proposition \ref{requiredProperties}--e), there is a null set $N_f$ such that $\overline I_jg(S)=0$ for every $S\in \Se \setminus N_f$. Consequently, for every $S\in \Se \setminus N_f$ and $\epsilon>0$ there are sequences $(V_m)_{m\geq 1}$  of nonnegative reals and $(H^m)_{m\geq 1}$ in $\He_{(S,j)}$ such that $\Pi_{j, n}^{V^m, H^m}\in \Ee^+_{(S,j)}$ for every $n\geq j$, $\sum_{m\geq 1} V_m\leq \epsilon$ and $g\leq \sum_{m\geq 1} \liminf_{n\to \infty}~~\Pi_{j, n}^{V^m, H^m}$ on $\Se_{(S, j)}$. Let $V^0=V(S_0,\ldots, S_j)$ and define $H^0$ via restricting the functions in $H$ on the conditional space $\Se_{(S, j)}$. Then,
	$$
	f\leq \Pi_{j,n_f}^{V^0,H^0}+\sum_{m\geq 1} \liminf_{n\to \infty}~~\Pi_{j, n}^{V^m, H^m} \mbox{ on } \Se_{(S, j)},
	$$
	which implies $\overline \sigma_j f(S)\leq V(S_0,\ldots,S_j)+\epsilon$ for every $S\in \Se \setminus N_f$. By passing with $\epsilon$ to zero, we obtain $\overline \sigma_j f\leq V$ a.e.   
\end{proof}

\begin{remark}
	Example \ref{exmp:no_martingale_measure} in Section \ref{sec:(L)} provides an example of a nonnegative, bounded function $f$ with finite maturity on a trajectory set satisfying $(L)$-a.e. such that $\overline I(f)>\overline\sigma f$. In light of the previous theorem, we may conclude that $\bar I$ cannot be applied for computing superhedging prices in general, but only serves as an auxiliary operator to determine the null sets of the model.
\end{remark}

The following theorem shows that inconsistencies between the two families of superhedging operators stem from the failure of $(L_{(S,j)})$ on a null set.

\begin{theorem}\label{thm:K}
	The following assertions are equivalent:\\[0.1cm]
	(i) For every $f\in P$ with finite maturity and every node $(S,j)$,
	$$
	\overline I_jf(S)=\overline \sigma_jf(S).
	$$
	(ii) $(L_{(S,j)})$ holds at every node $(S,j)$.
\end{theorem}

\begin{proof}
	(i) $\Rightarrow$ (ii): Since $\overline I_j(0)(S)=0$ always holds, we immediately obtain 
	$$
	\overline \sigma_j(0)(S)=\overline I_j(0)(S)= 0
	$$
	at any node $(S,j)$.
		\\[0.1cm]
	(ii) $\Rightarrow$ (i): \emph{Step 1:} We first consider the initial node $(S,0)$:
	
	Noting that $\overline \sigma f \leq \overline I f$ by definition, it suffices to show that $\overline I f \leq \overline  \sigma  f$, for which we may and do assume that $\overline \sigma f<\infty$. 
	  We now proceed as in the first part of the proof of Theorem \ref{thm:sigma_is_correct}, but note that the supermartingale $(f_i)_{i\geq 0}$ can be chosen $[0,\infty)$-valued and that Lemma \ref{lem:positive_martingale} is applicable, because $(L_{(S,j)})$ holds at every node $(S,j)$. By Theorem~\ref{thm:decomposition} and Lemma \ref{lem:positive_martingale}, we find 
	a non-anticipative sequence $(H_j)_{j=0,\ldots, n_f-1}$ such that
	$$
	\left( \overline \sigma f+\epsilon\right) +\sum_{i=0}^{\min\{j,n_f-1\}}H_i(S)\Delta_iS \geq 0 
	$$
	for every $S\in \mathcal{S}$ and 
	$$
	f(S)\leq \left( \overline \sigma f+\epsilon\right) +\sum_{i=0}^{n_f-1}H_i(S)\Delta_iS 
	$$
	for every $S\in \Se\setminus N_f$, where $N_f$ is a null set. Let $V^0=\overline \sigma f+\epsilon$ and define $H^0$ via $H^0_j=H_j$ for $j\leq n_f-1$ and $H^0_j=0$ for $j\geq n_f$.  Then $\Pi_{0,j}^{V^0,H^0}\in \mathcal{E}_0^+$ for every $j\geq 0$. Dealing with the null set $N_f$ as in the second part of the proof of Theorem \ref{thm:sigma_is_correct}, we conclude that $\overline I f\leq \overline \sigma f+2\epsilon$. Letting $\epsilon$ tend to zero, the proof of Step 1 is complete.
	\\[0.1cm]
	\emph{Step 2:}
	We now consider a generic node $(\tilde S,\tilde j)$: 
	
	Define the auxiliary trajectory set
	$$
	\widetilde{\mathcal{S}}=\{(S_{\tilde j+i})_{i\geq 0}|S\in \mathcal{S}_{(\tilde S,\tilde j)}\}.
	$$
	Then, the $\overline \sigma$-operator and the $\overline I$-operator  for $\widetilde{\mathcal{S}}$ at time 0 coincide with $\overline \sigma_{\tilde j}(\cdot)(\tilde S)$ and  $\overline I_{\tilde j}(\cdot)(\tilde S)$. Moreover, each node $(S,j)$ in $\widetilde{\mathcal{S}}$ coincides with the node $((\tilde S_0,\ldots,\tilde S_{\tilde j-1}, S_0,\ldots),$ $\tilde j+j)$ in $\mathcal{S}$. Hence, every node $(S,j)$ in $\widetilde{\mathcal{S}}$ satisfies $(L_{(S,j)})$. These observations reduce the case of a generic node to the case of an initial node.
\end{proof}

\section{Property $(L)$-a.e. and Some Examples} \label{sec:(L)}

In this section, we provide some easy-to-check sufficient conditions for the Assumption $(L)$-a.e., which is crucial for the main results of this paper, and for the additional assumption (H.1) required in the Supermartingale Convergence Theorem \ref{doobConvergence}. We will also present some examples.

We first recall the notion of trajectorial completeness from \cite{ferrando}.  
Given a sequence $(S^n)_{n\ge 0} \subseteq \mathcal{S}$ satisfying
\begin{equation}  \nonumber 
	~~~S^n_i = S^{n+1}_i \;\; 0 \leq i \leq n, ~~\forall ~n,~~~~~~~~
\end{equation}
define
\begin{equation}  \nonumber
	\overline{S}= (\overline{S}_i)_{i \geq 0}~~\mbox{by}~~~\overline{S}_i \equiv S^i_i.~~~
	\mbox{We will use the notation}~~\overline{S}= \lim_{n \rightarrow \infty} S^n.
\end{equation}
Notice that 
\begin{equation} \label{limitSequence}
	\overline{S}_i = S^n_i, ~~~0 \leq i \leq n,  ~~\forall ~n \geq 0.
\end{equation}
Let $\overline{\Se}$ be the set of such $\overline{S}$. Then, 
clearly, $\Se \subseteq \overline{\Se}$ given that for $\tilde{S}$ we can take $\tilde{S}^n = \tilde{S}$ for all $n \geq 0$. We say that $\mathcal{S}$ is {\it trajectorially complete}, if
$\mathcal{S}= \overline{\mathcal{S}}$.

As argued in Proposition 13 of \cite{ferrando}, $\overline{\mathcal{S}}$ is always trajectorially complete. Moreover, the completion process, i.e. passing from $\mathcal{S}$ to $\overline{\mathcal{S}}$ does not alter the type of the nodes (being up-down, no arbitrage, etc.), but it can change a null node (i.e., a node which constitutes a null set) into a non-null node.

The proof of the following theorem is similar to the arguments leading to Corollary 4 in \cite{ferrando}, where $\mathcal{S}$ is, however, assumed to have up-down nodes only. We provide the proof in Appendix \ref{app:L}, in which we also discuss more general sufficient conditions for $(L)$-a.e.

\begin{theorem}[{\bf  Sufficient conditions for $(L)$-a.e.}]\label{thm:L}
	Suppose $\mathcal{S}$ is trajectorially complete and the following condition holds:
	\begin{itemize}
		\item[(H.2)] If $(S,j)$ is an arbitrage node of type II, then $j\geq 1$, $(S,j-1)$ is an up-down node and for every $\epsilon>0$ there are $S^{\epsilon,1}, S^{\epsilon,2}\in \Se_{(S,j-1)}$ such that 
		$$
		S^{\epsilon,1}_j-S_{j-1}\geq -\epsilon,\quad S^{\epsilon,2}_j-S_{j-1}\leq \epsilon
		$$
		and such that $(S^{\epsilon,1},j)$, $(S^{\epsilon,2},j)$ are not type II arbitrage nodes.
	\end{itemize}
	Then, $(L)$-a.e. holds. More, precisely, $(L_{(S,j)})$ fails at a node $(S,j)$, if and only if $(S,j)$ is an arbitrage node of type II.
\end{theorem}

Whenever the trajectory set $\mathcal{S}$ is trajectorially complete and has no arbitrage nodes of type II, it follows that Theorem \ref{thm:L} implies that $(L_{(S,j)})$ holds at any node $(S,j)$ -- and then  condition (H.1), which is required in the supermartingale convergence theorem, is trivially satisfied. The assumption of trajectorial completeness  and the validity of $(L_{(S,j)})$  for all nodes $(S,j)$ is an implicit assumption in \cite{shafer} (see also Remark \ref{rem:ShaferVovk}). We weaken completeness in Theorem \ref{proofOfL} by means of hypothesis (H.5) and illustrated in Remark \ref{incompleteExample}.
	
The following corollary provides sufficient conditions for $(H.1)$ in the presence of arbitrage nodes of type II.

\begin{corollary}\label{cor:L}
	Suppose $\mathcal{S}$ is trajectorially complete and the following condition holds:
	\begin{itemize}
		\item[(H.3)] If $(S,j)$ is an arbitrage node of type II, then $j\geq 1$, $(S,j-1)$ is an up-down node and there are $S^{1}, S^{2}\in \Se_{(S,j-1)}$ such that 
		$$
		S^{1}_j> S_{j}> S^{2}_j
		$$
		and such that $(S^{1},j)$, $(S^{2},j)$ are not type II arbitrage nodes.
	\end{itemize}
	Then, $(L)$-a.e. and condition (H.1) in Theorem \ref{doobConvergence} hold.	
\end{corollary}
\begin{proof}
	We first show that (H.3) implies condition (H.2) in Theorem \ref{thm:L}. Suppose $(S,j)$ is a type II arbitrage node. Then, by (H.3), $j\geq 1$ and $(S,j-1)$ is an up-down node. Consequently, there is a $\tilde S\in (S,j-1)$ such that $\tilde S_j> \tilde S_{j-1}=S_{j-1}$. Then, for every $\epsilon>0$
	$$
	\tilde S_j-S_{j-1}>0>-\epsilon;
	$$
	and we may take $S^{\epsilon,1}=\tilde S$ in (H.2), provided $(\tilde S,j)$ is not an arbitrage node of type II. Otherwise, we may apply (H.3) with $\tilde S$ in place of $S$ and find some $S^1\in (S,j-1)=(\tilde S,j-1)$ such that $S^1_j>\tilde S_j$ and $(S^1,j)$ is not an arbitrage node of type II. Then, we may take $S^{\epsilon,1}=S^1$. The construction of $S^{\epsilon,2}$ follows by a `symmetric' argument.
	
	Now that (H.2) is verified, Theorem \ref{thm:L} implies that $(L)$-a.e. holds -- and that $L_{(S,j)}$ fails, if and only if $(S,j)$ is an arbitrage node of type II. The latter implies (H.1). Indeed, if $(S,j)$ is an up-down node and $L_{(S,j+1)}$ fails, then $(S,j+1)$ is a type II arbitrage node, and, thus, by (H.3) there are $S^{1}, S^{2}\in \Se_{(S,j)}$ such that 
	$$
	S^{1}_{j+1}> S_{j+1}> S^{2}_{j+1}
	$$
	and such that $L_{(S^1,j+1)}$, $L_{(S^2,j+1)}$ hold, because  $(S^{1},j)$, $(S^{2},j)$ are not type II arbitrage nodes.
\end{proof}

The following example illustrates that condition $(L)$-a.e. may hold in the absence of martingale measures and that that the $\delta$-sequence in the supermartingale decomposition (Theorem \ref{thm:decomposition}) cannot be dispensed with. At the same time, the example shows that the operators $\overline \sigma$ and $\overline I$ do, in general, not coincide on the cone of nonnegative functions.

\begin{ej}\label{exmp:no_martingale_measure}
Let
$$
\mathcal{S}=\mathcal{S}^+\cup  \mathcal{S}^-,\quad \mathcal{S}^\pm=\{S^{\pm,n}|n\in \mathbb{N}\}
$$	
where 
$$
S^{+,n}_i=\begin{cases} 1, & i=0 \\ 2, & i=1 \\ 2+\frac{1}{n}, & i\geq 2\end{cases}, \quad\quad  S^{-,n}_i=\begin{cases} 1, & i=0 \\ 1-\frac{1}{n^2}, & i\geq 1 \end{cases}.
$$
Then, the node $(S^{+,1},1)=(S^{+,n},1)$, $n\geq 2$, is an arbitrage node of type II, the initial node $(S,0)$ is up-down, and all other nodes are flat. Trajectories are displayed in the following illustration:
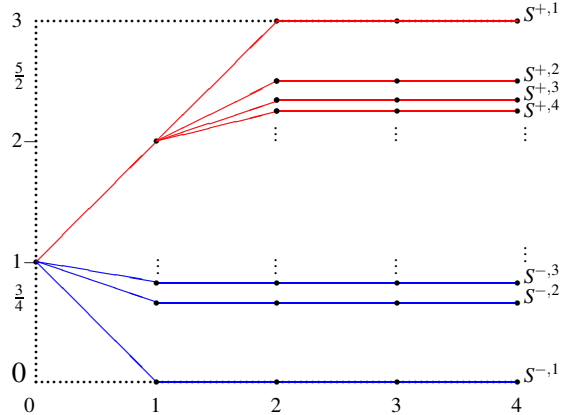
\[
{\setlength{\unitlength}{.08 cm} 
	\begin{picture}(80,65)
		\multiput(0,0)(0,1){60}{\circle*{.15}}
		\multiput(0,0)(1,0){80}{\circle*{.25}}
		\multiput(0,60)(1,0){80}{\circle*{.25}}
		\put(-4,0){\small $0$}\put(-4,13){\tiny $\frac34$}\put(-4,19){\tiny $1-$}
		\put(-4,39){\tiny $2-$}\put(-4,50){\tiny $\frac52$}\put(-4,59.05){\tiny $3$}
		\multiput(40,60)(20,0){3}{\circle*{1}}\put(81,60){\tiny $S^{+,1}$}
		\put(0,20){\circle*{1}}\put(0,20){\color{red}\line(1,1){20}}\put(20,40){\circle*{1}}
		\put(20,40){\color{red}\line(1,1){20}}\put(40,60){\color{red}\line(1,0){40}}
		\put(20,40){\color{red}\line(2,1){20}}\put(40,50){\circle*{1}}\put(40,50){\color{red}\line(1,0){40}} \multiput(40,50)(20,0){3}{\circle*{1}}\put(81,50){\tiny$S^{+,2}$}
		\put(20,40){\color{red}\line(3,1){20}}\put(40,46.9){\circle*{1}}\put(40,46.9){\color{red}\line(1,0){40}}
		\multiput(40,46.9)(20,0){3}{\circle*{1}}\put(81,46.9){\tiny$S^{+,3}$}
		\put(20,40){\color{red}\line(4,1){20}}\put(40,45){\circle*{1}}\put(40,45){\color{red}\line(1,0){40}}
		\multiput(40,45)(20,0){3}{\circle*{1}}\put(81,44){\tiny$S^{+,4}$} 
		\put(0,20){\color{blue}\line(6,-1){20}}\put(20,16.5){\color{blue}\line(1,0){60}}\multiput(20,16.5)(20,0){4}{\circle*{1}}
		\put(81,16.5){\tiny $S^{-,3}$}
		\put(0,20){\color{blue}\line(3,-1){20}}\put(20,13.15){\color{blue}\line(1,0){60}}\multiput(20,13.15)(20,0){4}{\circle*{1}}
		\put(81,13.15){\tiny $S^{-,2}$}
		\put(0,20){\color{blue}\line(1,-1){20}}\put(20,0){\color{blue}\line(1,0){60}}\multiput(20,0)(20,0){4}{\circle*{1}}
		\put(81,0){\tiny $S^{-,1}$}
		\put(-2,-5){\tiny $0$}\put(19,-5){\tiny $1$}\put(39,-5){\tiny $2$}\put(59,-5){\tiny $3$}\put(79,-5){\tiny $4$}
		\put(81,40){\scalebox{.6}{$\vdots$}} \put(39.5,40){\scalebox{.6}{$\vdots$}} \put(59.5,40){\scalebox{.6}{$\vdots$}} \put(81,20){\scalebox{.6}{$\vdots$}} \put(20,18){\scalebox{.6}{$\vdots$}} \put(39.5,18){\scalebox{.6}{$\vdots$}} \put(59.5,18){\scalebox{.6}{$\vdots$}} 
\end{picture}}
\]

\
\\[0.2cm]
a) We first show that $(L)$-a.e. holds, but there is no martingale measure for this trajectory set.

 Since $S^{-,n}_1-S^{-,n}_0=-1/n^2$ is arbitrarily close to zero for sufficiently large $n$, we observe that (H.2) is satisfied. Trajectorial completeness is obvious, because all trajectories stay constant after time $j=2$. Hence, by Theorem~\ref{thm:L}, this trajectory set satisfies $(L)$-a.e. Since all trajectories $S^{+,n}$ pass through the arbitrage node $(S^{+,1},1)$ of type II, $\mathcal{S}^+$ is a null set by Lemma \ref{nullityOfN_j} and, hence, $\bar I({\bf 1_{ \mathcal{S}^-}})=1$ by Remark \ref{L0ForNonTrivialityOfComplement}.

It is also clear that there is no probability measure $Q$ on the power set of $\mathcal{S}$, which turns the coordinate process $T_j:\mathcal{S}\rightarrow \mathbb{R}$, $S\mapsto S_j$ into a martingale. Otherwise $Q(\mathcal{S}^+)=0$ (because any martingale measure assigns probability zero to type II arbitrage nodes) and then $T_1<T_0$ $Q$-almost surely -- a contradiction. 

Consequently, by Example \ref{ex:supermart}--a), the coordinate process $(T_j)_{j\geq 0}$ is an example of a trajectorial martingale, which fails to be a martingale in the classical probabilistic setup.
\\[0.2cm]
b) We next construct a supermartingale, for which a decomposition as in Theorem \ref{thm:decomposition}--(ii) is not possible, if we let $\delta_j\equiv 0$ (and, hence, the small $\delta$-errors cannot be avoided in the formulation of the supermartingale decomposition theorem).

To this end, we define $ f_j:\mathcal{S} \rightarrow \mathbb{R}$ via
$$
f(S)=\begin{cases}
	0, & S \in \mathcal{S}^+ \\ \frac{1}{n}, & S=S^{-,n}
\end{cases}, \quad \quad 
f_j(S)=\begin{cases}
	0, & j=0 \\ f(S), & j\geq 1,
\end{cases}
$$
and consider the sequence $(f_j)_{j\geq 0}$.  Since all of its `paths' $j\mapsto f_j(S)$ are nondecreasing, $(f_j)$ obviously is a submartingale.
We claim that $(f_j)_{j\geq 0}$ is a also supermartingale (despite of the nondecreasing paths)  and apply Theorem \ref{thm:decomposition} to verify this. Given a sequence $(\delta_j)_{j\geq 0}$ of positive reals, let 
$$
H_0\equiv -\lceil {\delta_0}^{-1} \rceil,\quad H_j\equiv 0, \; j\geq0,
$$
and note that, for every $n\in \mathbb{N}$,
$$
\delta_0+H_0(S^{-,n}_1-S^{-,n}_0)= \delta_0 + \lceil {\delta_0}^{-1} \rceil \; \frac{1}{n^2} \geq \frac{1}{n}=f_1(S^{-,n}),
$$
by considering the cases $n\leq \lceil {\delta_0}^{-1} \rceil$ and $n>\lceil {\delta_0}^{-1} \rceil$ separately. Hence, we may define a nondecreasing, nonanticipating sequence $(A_j)$ via $A_0\equiv 0$ and
$$
A_{j}(S)-A_{j-1}(S)=\begin{cases}  \delta_0 + \lceil {\delta_0}^{-1} \rceil \cdot \frac{1}{n^2} - \frac{1}{n}, & j=1 \textnormal{ and } S=S^{-,n}  \\ \delta_{j-1}, & j\geq 2 \textnormal{ or } S\in \mathcal{S}^+.  \end{cases}
$$
It is then straightforward to check that, for every $i\geq 0$ and $S\in \mathcal{S}^-$,
\begin{equation*}\label{eq:ex2}
	f_i(S)=f_0+\sum_{j=0}^{i-1}H_j(S)\Delta_jS-A_i(S)+\sum_{j=0}^{i-1}\delta_j.
	\end{equation*}
Hence, $(f_j)$ is a supermartingale by Theorem \ref{thm:decomposition}. Note that, in view of Remark \ref{rem:supermartingale}, $(f_j)_{j\geq 0}$ is a martingale.

We finally prove that a decomposition as in Theorem \ref{thm:decomposition}--(ii) is not possible for this (super-)martingale $(f_j)$, if we let $\delta_j\equiv 0$. We first show that for every $H_0\in \mathbb{R}$ there is an $n_0\in \mathbb{N}$ such that for every $n\geq n_0$
$$
f_1(S^{-,n})>f_0+H_0(S^{-,n}_1-S^{-,n}_0).
$$
By inserting the definition of $f_1$ and $f_0$, this inequality is equivalent to 
$$
\frac{1}{n}\left(1-\frac{H_0}{n} \right)>0,
$$ 
which trivially holds for sufficiently large $n$. Hence, no matter of the choice of $(H_j)$ and $(A_j)$, a decomposition as in Theorem \ref{thm:decomposition}--(ii) with $\delta_0= 0$ will necessarily fail at $i=1$ on the set $\mathcal{S}^{-,\geq n_0}=\{S^{-,n}|n\geq n_0\}$ for some $n_0\in \mathbb{N}$. It, thus, remains to show that $\bar I({\bf 1}_{\mathcal{S}^{-,\geq n_0}})>0$ for every $n_0\in \mathbb{N}$.
To this end, assume that for every $S\in \mathcal{S}$,
$$
{\bf 1}_{\mathcal{S}^{-,\geq n_0}}(S) \leq  \sum_{m \geq 1} \liminf_{k\to \infty}~~\Pi_{0, k}^{V^m, H^m}(S)
$$
where $\Pi_{0, k}^{V^m, H^m}\in \mathcal{E}_0^+$ for all $k\geq 0$ and $\sum_{m\geq 1} V_m<\infty$. Since the trajectories in $\mathcal{S}^-$ are constant after time $k=1$, the previous inequality and the Aggregation Lemma \ref{convergenceOfPortfolioCoordinates} imply
$$
1\leq \sum_{m \geq 1} \Pi_{0, 1}^{V^m, H^m}(S^{-,n})=\sum_{m\geq 1} (V^m-H^m_0\frac{1}{n^2})= \sum_{m\geq 1} V_m -\frac{1}{n^2} \sum_{m\geq 1} H^m_0
$$
for every $n\geq n_0$. Passing with $n$ to infinity, we observe that $\sum_{m\geq 1} V_m\geq 1$ and, hence, $\bar I({\bf 1}_{\mathcal{S}^{-,\geq n_0}})=1$ for every $n_0\in \mathbb{N}$.

The foregoing also shows that $(f_j)_{j\geq 0}$ is an example of a martingale, which is not a.e. equal to a `simple' martingale of the form discussed in Example \ref{ex:supermart}--a).
\\[0.2cm]
c) We finally show that for $f$ (defined in b)): $\overline\sigma f=0 <1/2=\overline If$.

Since $(f_j)$ is a martingale, we obtain $0=f_0=\overline \sigma f_1=\overline \sigma f$. For the claim concerning $\overline I $, it clearly suffices to show: If $\overline I f<\epsilon$ for some $\epsilon>0$, then $\epsilon\geq 1/2$. If the hypothesis is satisfied, then, by the Aggregation Lemma \ref{convergenceOfPortfolioCoordinates}  at time 0, there is a constant $H_0$ such that $\epsilon +H_0(S_1-S_0)\geq 0$ for every $S\in \mathcal{S}$ and  $\epsilon +H_0(S_1-S_0)\geq f(S)$, whenever $S=S^{-,n}$ (because the trajectories in the down-branch stay constant after time 1). If we apply these inequalities to the trajectories $S^{-,1}$ and $S^{+,1}$, we obtain $\epsilon-H_0\geq 1$ and $\epsilon+H_0\geq 0$. Hence, $\epsilon\geq 1/2$.
\end{ej}

\begin{remark}
Consider the following variant	of Example \ref{exmp:no_martingale_measure}. Let $\mathcal{S}=\mathcal{S}^+\cup\{S^0,S^-\}$, where the up-branch $\mathcal{S}^+$ is as in Example \ref{exmp:no_martingale_measure}, $S^0\equiv 1$, $S^{-}_0=1$ and $S^-_j=0$ for $j\geq 1$. Adapting the arguments in Example \ref{exmp:no_martingale_measure}, one can check that: 1) The point mass $Q$ on $S^0$ is the unique martingale measure of this model; 2) $(L)$-a.e. holds; 3) $\bar I({\bf 1}_{\{S^0\}})=1$ and  $\bar I({\bf 1}_{\{S^-\}})=1/2$. 

Hence $\{S^-\}$ is a null set for $Q$, but not w.r.t. $\bar I$. In such a situation our supermartingale decomposition (Theorem \ref{thm:decomposition}) holds on a larger set than the (uniform) Doob decomposition in the classical theory.
\end{remark}

\begin{remark} \label{incompleteExample}
We may also consider the following variant of Example \ref{exmp:no_martingale_measure}, replacing the `sure' arbitrage at the type II arbitrage node $(S^+,1)$ by a `sure' arbitrage opportunity by trading up to unbounded time. To this end, we replace the up-branch of the model by $\tilde{\mathcal{S}}^+=\{\tilde S^{+,n}|n\in \mathbb{N} \}$, where now 
$$
\tilde S^{+,n}_i=\begin{cases} 1, & i=0 \\ 2, & i=1<n+1 \\ 4, & i\geq n+1 \end{cases},
$$
and consider the trajectory set $\tilde{\mathcal{S}}^+\cup \mathcal{S}^-$, with the lower branch $\mathcal{S}^-$ defined as in Example \ref{exmp:no_martingale_measure}. This modified trajectory set fails to be trajectorial complete, since $(1,2,2,2,\ldots)\notin \tilde{\mathcal{S}}^+\cup \mathcal{S}^-$, but satisfies $(L)$-a.e. by Corollary \ref{cor:L-a.e.}. Note that all potential losses at the node $(\tilde S^{+,1},1)$ by trading between time $0$ and $1$  can be  recuperated by buying the stock (at all times $i\geq 1$) and waiting until the stock price eventually increases to 4. All conclusions of Example \ref{exmp:no_martingale_measure} are easily seen to remain valid for this variant of the trajectory set. Note, however, that the completion of this trajectory set, i.e. adding the trajectory $(1,2,2,2,\ldots)$, drastically changes the model. Recuperating losses after time 1, is not possible anymore in the completed model, because the stock price may continue constantly after time 1. For this reason, the completed up-branch is not a null set anymore and the completed model has infinitely many martingale measures (all of which assigning positive probability on the `new' trajectory $(1,2,2,2,\ldots)$).   
\end{remark}

We finally present an example, in which $(L)$-a.e. fails, and demonstrate the importance of this assumption for our results.
\begin{ej}
	Let $\mathcal{S}=\{S^{+,-}, S^0, S^-\}\cup\{S^{+,n}|\, n\in \mathbb{N}\}$, where $S_0=1$ for every $S\in \mathcal{S}$,
	$$
	S_1=\begin{cases} 1, & S=S^0 \\ 2, & S\in \{S^{+,n}, S^{+,-}|\,n \in \mathbb{N}\} \\  0, & S=S^- \end{cases},\quad 	S_2=\begin{cases} S_1, & S\in \{S^0, S^-\} \\ 3, & S\in \{S^{+,n}|\,n \in \mathbb{N}\} \\  3/2, & S=S^{+,-} \end{cases}
	$$
	and 
	$$
		S_j=\begin{cases} S_2 , & S\in \{S^0, S^-, S^{+,-}\}  \\ 3+1/n, & S=S^{+,n} \end{cases},\quad j\geq 3.
	$$
	\
	This trajectory space is illustrated in the following figure:
	\
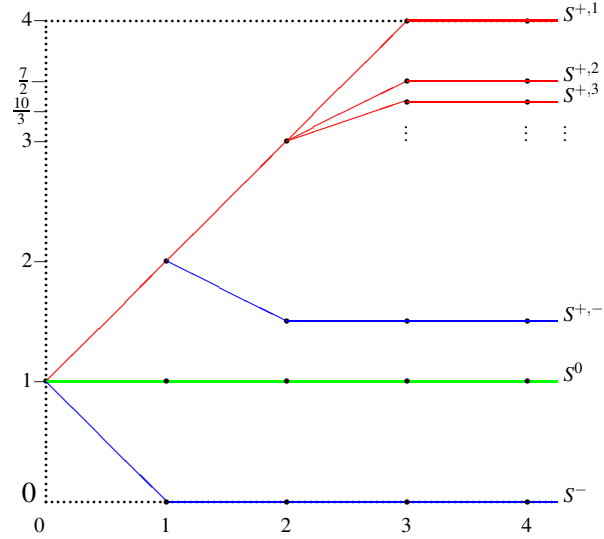
	\[
	{\setlength{\unitlength}{.08 cm} 
		\begin{picture}(80,85)
			\multiput(0,0)(0,1){80}{\circle*{.15}}
			\multiput(0,0)(1,0){85}{\circle*{.25}}
			\multiput(0,80)(1,0){85}{\circle*{.25}}
			\put(-4,0){\small $0$}\put(-4,19){\tiny $1-$}
			\put(-4,39){\tiny $2-$}\put(-4,59.05){\tiny $3-$}\put(-5,69){\tiny ${\frac72}-$}
			\put(-6,64){\tiny ${\frac{10}3}-$}\put(-4,79.05){\tiny $4-$}
			\put(0,20){\circle*{1}}\put(0,20){\color{red}\line(1,1){20}}\put(20,40){\circle*{1}}
			\put(20,40){\color{red}\line(1,1){20}}\put(40,60){\color{red}\line(1,1){20}}\put(40,60){\circle*{1}}
			\put(60,80){\color{red}\line(1,0){25}}\multiput(60,80)(20,0){2}{\circle*{1}}\put(86,80){\tiny $S^{+,1}$}
			\put(40,60){\color{red}\line(2,1){20}}\multiput(60,70)(20,0){2}{\circle*{1}}\put(60,80){\color{red}\line(1,0){25}}
			\put(60,70){\color{red}\line(1,0){25}}\put(86,70){\tiny$S^{+,2}$}
			\put(40,60){\color{red}\line(3,1){20}}\multiput(60,66.5)(20,0){2}{\circle*{1}}\put(60,66.5){\color{red}\line(1,0){25}}
			\put(86,66.5){\tiny$S^{+,3}$}
			\put(59.5,60){\scalebox{.6}{$\vdots$}} \put(79.5,60){\scalebox{.6}{$\vdots$}} 
			\put(86,60){\scalebox{.6}{$\vdots$}} 
			\put(20,40){\color{blue}\line(2,-1){20}}\multiput(40,30)(20,0){3}{\circle*{1}}\put(40,30){\color{blue}\line(1,0){45}}
			\put(86,30){\tiny$S^{+,-}$}
			\put(0,20){\color{green}\line(1,0){85}}\multiput(20,20)(20,0){4}{\circle*{1}}\put(86,20){\tiny$S^{0}$}
			\put(0,20){\color{blue}\line(1,-1){20}}\put(20,0){\color{blue}\line(1,0){65}}\multiput(20,0)(20,0){4}{\circle*{1}}
			\put(86,0){\tiny $S^{-}$}
			\put(-2,-5){\tiny $0$}\put(19,-5){\tiny $1$}\put(39,-5){\tiny $2$}\put(59,-5){\tiny $3$}\put(79,-5){\tiny $4$}
	\end{picture}}
	\]
	\
	\\[0.1cm]	
	Note that $(L)$ holds, because the constant trajectory $S^0$ is included in the trajectory set. \\[0.1cm]
	a) We first show that $(L_{(S^{+,-},1)})$ fails and $\bar I({\bf 1}_{\{S^{+,-}\}})\geq 1/6$, and, thus, condition $(L)$-a.e. is violated.
	
	Suppose $f:\mathcal{S}_{(S^{+,-},1)} \rightarrow \mathbb{R}$. Define, for arbitrary, but fixed $k\in \mathbb{N}$,
	$$
	g_0(S)=-k-2(f(S^{+,-})+k)(S_2-S_1),\quad g_{m}(S)=S_3-S_2,\quad S\in \mathcal{S}_{(S^{+,-},1)} ,\;m\geq 1.
	$$
	Then, $g_0 \in \mathcal{E}_{(S^{+,-},1)}$, $g_m\in \mathcal{E}^+_{(S^{+,-},1)}$ for $m\geq 1$, and, on $\mathcal{S}_{(S^{+,-},1)}$,
	$$
	\sum_{m=0}^\infty g_m =f(S^{+,-}) + \infty\,{\bf 1}_{\{S^{+,n}|\, n\in \mathbb{N}\}}\geq f.
	$$
	Since $k$ was arbitrary, we obtain $\overline{\sigma}_1f(S^{+,-})=-\infty$. By choosing $f\equiv 0$, we observe that  $(L_{(S^{+,-},1)})$ fails. 
	
	We next show that  $\bar I({\bf 1}_{\{S^{+,-}\}})\geq 1/6$. 
	To this end, assume that for every $S\in \mathcal{S}$,
	$$
	{\bf 1}_{\{S^{+,-}\}}(S) \leq  \sum_{m \geq 1} \liminf_{k\to \infty}~~\Pi_{0, k}^{V^m, H^m}(S)
	$$
	where $\Pi_{0, k}^{V^m, H^m}\in \mathcal{E}_0^+$ for all $k\geq 0$ and $\sum_{m\geq 1} V_m<\infty$. Since $S^{+,-}$ stays constant after time 1 and the  $\Pi_{0, k}^{V^m, H^m}$'s are nonnegative at any time, we obtain
	\begin{equation}\label{eq:hilf0009}
		{\bf 1}_{\{S^{+,-}\}}(S) \leq  \sum_{m \geq 1} \Pi_{0, k}^{V^m, H^m}(S)
	\end{equation} 
	for every $k\geq 1$ and $S\in \mathcal{S}.$	
	 We now write $v=\sum_{m\geq 1} V_m$, $a=\sum_{m\geq 1} H^m_0$ and $b=\sum_{m\geq 1} H^m_1(S^{+,-})$ and note that the series defining $a$ and $b$ converge in $\mathbb{R}$ by the Aggregation Lemma \ref{convergenceOfPortfolioCoordinates}. Inserting the pairs $S=S^{+,-}$, $k=1$; $S=S^{+,1}$, $k=2$; $S=S^-$, $k=1$ into \eqref{eq:hilf0009} yields
	 $$
	 (I): \,v+a-b/2\geq 1,\quad (II):\,v+a+b\geq 0,\quad (III):\,v-a\geq 0.
	 $$
	 Considering $2(I)+(II)+3(III)$, we observe that $v\geq 1/6$ and, hence, $\bar I({\bf 1}_{\{S^{+,-}\}})\geq 1/6$. Note that by similar, but easier arguments $\bar I({\bf 1}_{\{S^{-}\}})\geq 1/2$ and $\bar I({\bf 1}_{\{S^{0}\}})=1$. 
	 \\[0.1cm]
	 b) We now define the sequence $(f_j)_{j\geq 0}$ via $f_0=1$ and  $f_j\equiv f_\infty={\bf 1}_{\{S^0\}}+2 \;{\bf 1}_{\mathcal{S}\setminus \{S^0\}}$ for $j\geq 1$.
	 We claim that $(f_j)_{j\geq 0}$ is a supermartingale, which does not have a decomposition as in Theorem \ref{thm:decomposition}--(ii) for sequences $(\delta_j)$ with $0<\delta_0<1$. In particular, the assumption $(L)$-a.e. cannot be dropped in the latter theorem.
	 
	 In our considerations, we may ignore the null set $\{S^{+,n}|\, n\in \mathbb{N}\}$ of those trajectories, which pass through the arbitrage node $(S^{+,1},2)$ of type II. As the computation of $\overline \sigma_j$ is trivial after trajectories have become constant, we get
	 $$
	 \overline \sigma_j f_{j+1}(S)=\overline \sigma_j f_{\infty}(S)=f_{\infty}(S)=f_{j}(S)
	 $$
	 for $S\in \{S^0,S^-\}$ and $j\geq 1$ and for $S=S^{+,-}$ and $j\geq 2$. Moreover, by a), $$ \overline{\sigma}_1f_2(S^{+,-})=-\infty\leq f_1(S^{+,-}).$$ 
	 For the supermartingale property at the initial node, consider
	 $$
		g_0(S)=1-(S_2-S_1)-4(S_2-S_1),\quad g_{m}(S)=S_3-S_2,\quad S\in\mathcal{S}, \;m\geq 1.
	 $$
	 Then, $g_0 \in \mathcal{E}$, $g_m\in \mathcal{E}^+$ for $m\geq 1$, and, on $\mathcal{S}$,
	 $$
	 \sum_{m\geq 0} g_m=f_1\, {\bf 1}_{\{S^0,S^-,S^{+,-}\}}+ \infty\,{\bf 1}_{\{S^{+,n}|\,n\in \mathbb{N}\}}\geq f_1.
	 $$
	 Hence $ \overline \sigma f_1 \leq 1$. (Since the trajectory $S^0$ is constant,  we obviously also obtain $ \overline \sigma f_1\geq f_1(S^0)=1$, i.e., $ \overline \sigma f_1 = 1$.)
	 
	 We now fix a sequence $(\delta_j)_{j\geq 0}$ of positive reals and assume existence of a representation for $(f_j)$ as in Theorem \ref{thm:decomposition}--(ii). Then, at time $i=1$,
	 $$
	 f_1(S)\leq (1+\delta_0)+H_0(S_1-S_0)
	 $$
	 for $S\in \{S^0,S^-, S^{+,-}\}$ (where $H_0$ is a constant), because none of the singletons $\{S\}$, $S\in \{S^0,S^-, S^{+,-}\}$, is a null set by a). This leads to the three inequalities
	 $$
	 1\leq (1+\delta_0),\quad 2 \leq (1+\delta_0)-H_0,\quad  2 \leq (1+\delta_0)+H_0,
	 $$
	 which imply $\delta_0\geq 1$.
	 \if
	 c) We finally show that condition $(L)$-a.e. can neither be dropped in the strong tower property of Corollary \ref{cor:tower}.
	 
	 We choose $f\equiv 1$, $\rho\equiv 0$, and $\tau\equiv 1$. Since $(L)$ holds, we obtain 
	 $\overline{\sigma}_\rho f(S)=\overline{\sigma} 1=1= \underline{\sigma} 1=\underline{\sigma}_\rho f(S)$ for every $S\in \mathcal{S}$ by Proposition \ref{Properties_L}, and, in particular, $\overline{\sigma}_\rho f\leq  \underline{\sigma}_\rho f$ a.e. Hence, (apart from $(L)$-a.e.), all assumptions of Corollary \ref{cor:tower} are satisfied. Since the trajectories stay constant after passing each of the nodes $(S^0,1)$ and $(S^-,1)$, we observe, thanks to a), that
	 $$
	 \overline{\sigma}_1f={\bf 1}_{\{S^0,S^-\}}-\infty{\bf 1}_{\mathcal{S}\setminus\{S^0,S^-\}}.
	 $$
	 Then, $\overline{\sigma}(-\overline{\sigma}_1f)=+\infty$. Indeed, if otherwise
	 $$
	 -\overline{\sigma}_1f(S) \leq  \sum_{m \geq 0} \liminf_{k\to \infty}~~\Pi_{0, k}^{V^m, H^m}(S)
	 $$
	 for every $S\in \mathcal{S}$, where $\Pi_{0, n_0}^{V^0, H^0}\in \mathcal{E}_0$, $\Pi_{0, k}^{V^m, H^m}\in \mathcal{E}_0^+$ for all $k\geq 0$ and $\sum_{m\geq 1} V_m<\infty$, then, 
	 by the Aggregation Lemma \ref{convergenceOfPortfolioCoordinates} at times $0$ and $1$, there are real constants $a,b$ such that
	 \begin{eqnarray*}
	 	+\infty=  -\overline{\sigma}_1f(S^{+,-})\leq \left( \sum_{m\geq 0} V_m\right)+a(S^{+,-}_1-S^{+,-}_0)+b(S^{+,-}_2-S^{+,-}_1)\in \mathbb{R},
	 \end{eqnarray*}
 a contradiction. Hence, for every $S\in \mathcal{S}$,
 $$
 \underline{\sigma}_\rho(\overline{\sigma}_\tau f)(S)=-\overline{\sigma}(-\overline{\sigma}_1 f)=-\infty \neq 1=\overline{\sigma}_\rho f(S),
 $$
 i.e., the assertion of Corollary \ref{cor:tower} is not valid for this example. 
 \fi
\end{ej}

\appendix

\section{Partitions, Arbitrage Nodes, and Null Sets}\label{Property L}
The following Lemma shows that property $(L_{(S,j)})$ fails at type II arbitrage nodes $(S,j)$.

\begin{lemma} \label{priceMinusInfinity}
	Given a trajectory set $\mathcal{S}$ consider  a  node $(S,j)$, $j \geq 0$, then:
	If $(S,j)$ is a type II arbitrage node, then
	\begin{equation}\nonumber 
		\overline{\sigma}_jf (S) = - \infty~\mbox{for any}~~f \in Q.
	\end{equation}
\end{lemma}

\begin{proof}
	We may consider the case when $\tilde{S}_{j+1} > S_j$ for all $\tilde{S} \in \mathcal{S}_{(S,j)}$. Take then, for all $m \geq 1$: $H^m_j(\tilde{S})=1$ and $H^m_i(\tilde{S})=0$
	for all $i >j$, $V^m =0$. Also, $H^0_i=0$ for all $i \geq j$; then, for any $V^0 \in \mathbb{R}$:
	\begin{equation} \nonumber
		f(\tilde{S}) \leq V^0 + \infty = V^0 + \sum_{m \geq 1} H^m_j(\tilde{S})~ \Delta_j \tilde{S}~~\hspace{.1in} \mbox{holds for any}~~\tilde{S} \in \mathcal{S}_{(S,j)}~~\mbox{and}~~f \in Q.
	\end{equation}
	Thus, the claim follows.
\end{proof}

In upcoming results it will be proved that trajectories passing through arbitrage nodes of type II form a null set. 
\vspace{.1in} \noindent Define for $j\ge 0$,
\[N^{\circ}_j \equiv \{S\in\Se: (S,j)\; \mbox{is an arbitrage node, and}\;\Delta_jS\neq 0\},\quad\]
\begin{equation} \label{someNullSets}
	N_k \equiv  \bigcup\limits_{j\ge k}N^{\circ}_j,~\mbox{for}~~ k \geq 0, ~~\mbox{and}~~~~N(S,j) \equiv N_j \cap \mathcal{S}_{(S,j)}~\mbox{for}~ j \geq 0.
\end{equation}
Notice that
$$
N_0=\mathcal{N}:=\mathcal{N}^{(I)}\cup \mathcal{N}^{(II)}=\{S\in\Se: \exists~~ j\ge 0 ~~~~~ \mbox{s.t.} \; (S,j)\; \mbox{is an arbitrage node and}\; S_{j+1}\neq S_j\}
$$
where the sets $\mathcal{N}^{(I)}, \mathcal{N}^{(II)}$ where introduced in the proof of Theorem \ref{thm:decomposition}.

It will be shown that $\mathcal{N}$ is a null set. Whenever $S\notin \mathcal{N}$ it follows that $S\notin N^{\circ}_j$ for any $j\ge 0$, therefore such node $(S,j)$ is:  flat, or up-down, or type I arbitrage node with $S_{j+1}=S_j$. On the other hand if $(S,j)$ is a type II arbitrage node then $\Se_{(S,j)}\subset N^{\circ}_j$. Moreover, it can be that $S\in N^{\circ}_j$, but $(S,k)$ is arbitrage free for some $k> j$.

\begin{definition}\label{indexSet}
	Since for any $j\ge 0$, $\Se$ is the disjoint union of $\Se_{(S,j)}$, let $\Lambda_j$ be an index set, such that for $\lambda \in \Lambda_j$ there exists $S^{\lambda}\in \Se$ such that
	\[
	\lambda \neq \lambda'\; \Rightarrow S^{\lambda'}\notin \Se_{(S^{\lambda},j)},\quad \Se =\bigcup_{\lambda \in \Lambda_j}\Se_{(S^{\lambda},j)},\quad \mbox{and}
	\]
	
	\vspace{-.15in}
	\[
	\mbox{if}~(S^{\lambda},j)~\mbox{is an arbitrage node then}\; |\Delta_jS^{\lambda}|>0.
	\]
	
	\vspace{.1in}For $\Gamma \subset \Lambda_j$ define $H^{\Gamma}=(H_i^{\Gamma})_{i\ge 0},\; H_i^{\Gamma}:\Se \rightarrow \mathbb{R}$ by
	\[
	H_i^{\Gamma}\equiv 0\;\;\mbox{if}\;\; i\neq j~;\quad H_j^{\Gamma}\equiv \mathbf{1}_{\Se^{\Gamma}}, \;\; \mbox{where}\;\; \Se^{\Gamma}\equiv \bigcup\limits_{\lambda \in \Gamma}\Se_{(S^{\lambda},j)}.
	\]
\end{definition}

$\mathbf{H}^{\Gamma}$ {\bf is non-anticipative:} Let $\tilde{S}_k=S_k,\; 0\le k\le i$. If $i\neq j$ then $H_i^{\Gamma}(\tilde{S})=H_i^{\Gamma}(S)=0$. For $i=j$, $S\in\Se_{(S^{\lambda},j)}$ iff $\tilde{S}\in\Se_{(S^{\lambda},j)}$ so $H_j^{\Gamma}(\tilde{S})=H_j^{\Gamma}(S)$.

\begin{lemma}\label{nullityOfN_j}
	Consider $j\ge 0$ and $0\le k\le j$, then $N^{\circ}_j$, thus also $N_j$, are conditionally null sets at any $(S,k)$.
\end{lemma}
\begin{proof} Define $N^{\circ,+}_j\equiv\{S\in N^{\circ}_j: \Delta_jS>0\}$ and $\Lambda_j^+\equiv\{\lambda\in\Lambda_j: S^{\lambda}\in N^{\circ,+}_j\}$, and consider for $m\ge 1$, $f_m=\Pi^{0,H^{\Lambda_j^+}}_{k,j+1}\in \mathcal{E}_{(S,k)}^+$ for any $S\in\Se$. Then
	\[\mathbf{1}_{N^{\circ,+}_j}\le \sum_{m\ge 1}f_m,\;\;\mbox{which implies that}\; \|\mathbf{1}_{N^{\circ,+}_j}\|_k=0.\]
	In a similar way it is shown that $N^{\circ,-}_j\equiv\{S\in N^{\circ}_j: \Delta_jS<0\}$ is a conditionally null set at any $(S,k)$, consequently $N^{\circ}_j=N_j^{\circ,+}\cup N_j^{\circ,-}$ is also conditionally null.
\end{proof}

\if
Lemma \ref{sttoping_disjoint_union} below gives a representation of a trajectory space $\Se$ as a disjoint union of subspaces $\Se_{(S,\tau(S))}$ where $\tau$ is a stopping time. It also allows to define global portfolios by pasting local ones.
\begin{lemma}\label{sttoping_disjoint_union}
	Let $\tau$ be a  stopping-time on a trajectory space $\Se$, then
	\begin{enumerate}
		\item \label{disjoint union} For $S^1,S^2\in \Se$, $\Se_{(S^1,\tau(S^1))}$ and $\Se_{(S^2,\tau(S^2))}$ are equal or disjoint.
		\item \label{pasting portfolios} For $S\in\Se$ let $H^S=(H^S_i)_{i\ge 0}\in \He_{(S,\tau(S))}$ with $H^S_i=0$ for $0\le i<\tau(S)$, and $H^{\tilde{S}}_i=H^S_i$ if $\tilde{S}\in \Se_{(S,\tau(S))}$ and $i\geq \tau(S)$. Define ${H_i}:\Se \to \mathbb{R}$ by ${H_i}_{|\Se_{(S,\tau(S))}}=H^S_i$ for $i\ge 0$. Then $H=(H_i)_{i\ge 0}\in\He$.
	\end{enumerate}
\end{lemma}
\begin{proof} For item (\ref{disjoint union}) assume there exists $S\in\Se_{(S^1,\tau(S^1))}\cap\Se_{(S^2,\tau(S^2))}$. Consider w.l.o.g. that $\tau(S^1)\le\tau(S^2)$. Then $S^1_i=S_i=S^2_i$ for $0\le i\le \tau(S^1)$, which implies hat $\tau(S^2)=\tau(S^1)$ and in consequence $\Se_{(S^1,\tau(S^1))}=\Se_{(S^2,\tau(S^2))}$.
	
	For item (\ref{pasting portfolios}), it is enough to prove that $H_i$ are non-anticipative for any $i\ge 0$. Assume $S^1_j=S^2_j$ for $0\le j\le i$ and w.l.o.g. $\tau(S^1)\le\tau(S^2)$.
	
	If $i<\tau(S^1)$ then $H_i(S^1)=H^{S^1}_i(S)=0=H^{S^2}_i(S^2)=H_i(S^2)$.
	
	While if $\tau(S^1)\le i$, then $\tau(S^2)=\tau(S^1)$ which leads to $S^2\in\Se_{(S^1,\tau(S^1))}$ and \\ $H_i(S^2)=H^{S^2}_i(S^2)=H^{S^1}_i(S^2)=H^{S^1}_i(S^1)=H_i(S^1)$.
\end{proof}
\fi

\section{Basic Properties of Superhedging Functional $\overline{\sigma}_j$}
\label{Paper 1}

\begin{proposition}[Basic Properties] \label{requiredProperties}
The following properties hold for $f,g \in Q$
\begin{itemize}
\item [a)] $\overline{\sigma}_jf(S)\le f(S)$ if $f$ is constant on  $\Se_{(S,j)}$. (Implies $\overline{\sigma}_j0\le 0$ and $\underline{\sigma}_jf\ge f$.)
\item [b)] $\overline{\sigma}_jf(S) \le \overline{\sigma}_jg(S)$, if $f\le g$ a.e. on $\mathcal{S}_{(S,j)}$.
\item [c)] $\overline{\sigma}_j[f+g]\le \overline{\sigma}_jf + \overline{\sigma}_jg$.
\item [d)] Let $f \in Q$,  $g \in P$, and $g$ is constant on $\mathcal{S}_{(S,j)}$ then
$\overline{\sigma}_j(g f)(S) \leq g \overline{\sigma}_j f(S)$.

\item [e)] Let $f \in P$ and $k \geq 0$ then $0 \leq \overline{I}(\overline{I}_k f)
\leq \overline{I}f$. Therefore is $f$ is a null function we get $\overline{I}_k f$ is a null function.
\end{itemize}
\end{proposition}
\begin{proof}
The proofs are immediate but we do indicate the arguments for item $d)$.

Let $c \geq 0$ be a constant such that $g(\hat{S}= c$ for all $\hat{S} \in \mathcal{S}_{(S,j)}$. If $c=0$  $\overline{\sigma}_j(g f)(S)= \overline{\sigma}_j(0)(S) \leq g \overline{\sigma}_j 0(S)$ (already covered by item $a)$. So assume $c >0$

Let $gf(\tilde{S}) \le \sum\limits_{m\ge 0}\liminf\limits_{n \rightarrow \infty}\Pi_{j, n}^{V^m, H^m}(\tilde{S}),\; \tilde{S}\in\Se_{(S,j)}$,
with $\Pi_{j, n_0}^{V^0, H^0}\in \mathcal{E}_{(S,j)}$ and, for $m\ge 1$, $\Pi_{j, n}^{V^m, H^m}\in \mathcal{E}^+_{(S,j)}$ for all $n\ge 0$.
For each $\tilde{S} \in \Se_{(S,j)}$ and ~~ $m\ge 0$ ~~~~ define
\[
U^m(\tilde{S}) = \frac {V^m}{g(S)},\;\;\mbox{and}\;\; G^m_i(\tilde{S})=\frac {H^m_i(\tilde{S})}{g(S)},\;\mbox{for}\; i\ge j.
\]
It follows that $f(\tilde{S})\le \sum\limits_{m\ge 0}\liminf\limits_{n \rightarrow \infty}\Pi_{j, n_m}^{U^m, G^m}(\tilde{S}),\; \tilde{S}\in\Se_{(S,j)}$, with $\Pi_{j, n_0}^{U^0, G^0}\in \mathcal{E}_{(S,j)}$, and for $m\ge 1$, $\Pi_{j, n}^{U^m, G^m}\in \mathcal{E}^+_{(S,j)}$ for all $n\ge 0$. Thus
\[
\overline{\sigma}_jf(S) \le \frac{\overline{\sigma}_j[gf](S)}{g(S)}.
\]
Notice that one actually obtains $\overline{\sigma}_j(g f)(S) = g \overline{\sigma}_j f(S)$
if $g=c>0$.

\vspace{.1in}
We also note that the proof of item $e)$ is analogous to the one of
Proposition\ref{mainDirectionOfTowerProperty s-t}.
\end{proof}

\begin{corollary}\label{cor:leq_ae}
	Suppose $f,g\in Q$. If $f\leq g$ a.e., then $\overline{\sigma}_j f \leq 	\overline{\sigma}_j g$ a.e.
\end{corollary}
\begin{proof}
Note that, for every $S\in \mathcal{S}$,
$$
\overline{\sigma}_j f(S)\leq \overline{\sigma}_j (\max\{f,g\})(S)= \overline{\sigma}_j (g+(f-g)_+)(S)\leq \overline{\sigma}_j g(S)+ \overline{I}_j(f-g)_+(S), 
$$	
where $(\cdot)_+$ denotes the positive part. As $(f-g)_+$ is a null function, we conclude by Proposition \ref{requiredProperties} that $\overline{I}_j(f-g)_+$ is a null function. Hence, $\overline{\sigma}_j f\leq \overline{\sigma}_j g$ a.e.
\end{proof}

Here are some equivalent properties for $(L_{(S,j)})$.
\begin{proposition} \label{Properties_L}
For a fixed node $(S,j)$, the following items are equivalent.
\begin{enumerate}
\item $\overline{\sigma}_j 0 (S) = 0$.\label{sigma0=0}
\item $\underline{\sigma}_jf(S)\le \overline{\sigma}_jf(S)$ for any $f\in Q$. \label{sigmaDown_Leq_sigmaUp}
\item Property $(L_{(S,j)})$.  \label{L_j}
\item $\underline{\sigma}_j f (S) = V(S)= \overline{\sigma}_j f(S)$ for $f = \Pi^{V, H}_{j, n_f} \in \mathcal{E}_j$. \label{sigma=I}
\end{enumerate}
\end{proposition}
\begin{proof}
From (\ref{sigma0=0}) and item c) of Proposition \ref{requiredProperties} it follows (2), since $0\le \overline{\sigma}_jf(S)+\overline{\sigma}_j[-f](S)$. Assumed (2), it follows that $0\le\underline{\sigma}_j0(S)\le \overline{\sigma}_j0(S)\le 0$, first and last inequalities from item a) of Proposition \ref{requiredProperties}, so (\ref{sigma0=0}) holds.

From here onwards we let $f=\Pi^{V, H}_{j, n_f}$. 

$(\ref{L_j})$ follows from $(\ref{sigma0=0})$ as follows. Let  $h_m = \liminf\limits_{n \rightarrow \infty}~ \Pi^{V^m, H^m}_{j, n},~~\Pi^{V^m, H^m}_{j, n} \in \mathcal{E}_{(S, j)}^+~\forall~ n \geq j$ and $m \geq 1$, such that $f \leq \sum\limits_{m \geq 1} h_m$.

\noindent Then $0 \leq -\Pi^{V, H}_{j, n_f} + \sum\limits_{m \geq 1} h_m$, thus (taking $f_0 \equiv -\Pi^{V, H}_{j, n_f}$,  and $f_m=h_m$ for $m\ge 1$) by Definition \ref{cond_integ_def}, $0 = \overline{\sigma}_j(0) \leq -V+\sum\limits_{m \geq 1} V^m$, which leads to $V \leq \sum\limits_{m \geq 1} V^m$ as required.

Assumed $(\ref{L_j})$, let $f_0=\Pi^{V^0, H^0}_{j, n_0}\in \mathcal{E}_{(S, j)}$ and $f_m = \liminf\limits_{n \rightarrow \infty}~ \Pi^{V^m, H^m}_{j, n},~~\Pi^{V^m, H^m}_{j, n} \in \mathcal{E}_{(S, j)}^+\\~\forall~ n \geq j$ and $m \geq 1$, such that $f \leq \sum\limits_{m \geq 0} f_m$. Then $f-f_0\le \sum\limits_{m \geq 1} f_m$ with $f-f_0\in \mathcal{E}_{(S, j)}$, so \[V(S)-V_0\le \sum\limits_{m \geq 0} V^m,\]
and \[V(S)\le \overline{\sigma}_jf(S).\]
Since by Definition \ref{cond_integ_def}, $\overline{\sigma}_jf(S)\le V(S)$, (\ref{sigma=I}) holds, having in mind that $\underline{\sigma}_jf(S)=-\overline{\sigma}_j[-f](S)= V(S)$.

Finally, it is clear that (\ref{sigma0=0}) follows from (\ref{sigma=I}).
\end{proof}

\begin{lemma} \label{integProperty}
	Let $f \in Q$, $(S,j)$ a fixed node and $k \ge j$. If $f$ is constant on $\mathcal{S}_{(S,j)}$ then $\overline{\sigma}_k f(S) \leq f(S)\le \underline{\sigma}_k f(S)$. Moreover once $(L_{(S,k)})$ holds then
	\begin{enumerate}
		\item \label{f_constant} If $f$ is constant on $\mathcal{S}_{(S,j)}$ then  $\underline{\sigma}_k f(S) = f(S) = \overline{\sigma}_k f(S)$.
		
		\item \label{general_f} For a general $f\in Q$, $\overline{\sigma}_j f$ is constant on $\mathcal{S}_{(S,j)}$, hence:
		\begin{equation} \nonumber
			\overline{\sigma}_k [\overline{\sigma}_j f](S) = \overline{\sigma}_j f(S) = \underline{\sigma}_k [\overline{\sigma}_j f](S)\quad \mbox{and} \quad \overline{\sigma}_k [\underline{\sigma}_j f](S)= \underline{\sigma}_j f(S) = \underline{\sigma}_k [\underline{\sigma}_j f](S).
		\end{equation}
	\end{enumerate}
\end{lemma}
\begin{proof}
	If $f$ is constant on $\mathcal{S}_{(S,j)}$, it is also constant on $\mathcal{S}_{(S,k)}\subset \mathcal{S}_{(S,j)}$, then $f\in \mathcal{E}_{(S,k)}$, so by Definition \ref{cond_integ_def},  $\overline{\sigma}_kf(S) \leq f(S)$.
	
	If, furthermore, also $(L_{(S,k)})$ holds,
	we have $\underline{\sigma}_kf(S) = I_k f(S)= f(S)=\overline{\sigma}_kf(S)$ by Proposition \ref{Properties_L} item (4).
\end{proof}

\begin{corollary}\label{linearityOfSigma}
Let $f, g \in Q$ and consider
a fixed $S\in\Se$. If for $j\ge 0$ property $(L_{(S,j)})$ holds and $\overline{\sigma}_j f(S) - \underline{\sigma}_j f(S) = 0 = \overline{\sigma}_j g(S) - \underline{\sigma}_j g(S)$, then all the involved quantities are finite and
\[
(a) \quad\overline{\sigma}_j(f+g)(S)= \overline{\sigma}_jf(S)+\overline{\sigma}_jg(S)= \underline{\sigma}_jf(S)+\underline{\sigma}_jg(S)=\underline{\sigma}_j(f+g)(S).
\]
\[
\hspace{-.9in} (b) \quad\overline{\sigma}_j(cf)(S)= c\overline{\sigma}_jf(S)=c\underline{\sigma}_jf(S)=\underline{\sigma}_j(cf)(S)\quad\forall c\in \mathbb{R}.
\]
\end{corollary}
\begin{proof}
The finiteness claims follow from our conventions in the first paragraph of Section \ref{a.e. section}. We then see that the  hypotheses imply that $\overline{\sigma}_j f(S)= \underline{\sigma}_j f(S)$ and $\overline{\sigma}_j g(S)= \underline{\sigma}_j g(S)$.

(a) holds from
\[
\overline{\sigma}_jf(S)+\overline{\sigma}_jg(S)=\underline{\sigma}_jf(S)+\underline{\sigma}_jg(S)\le \underline{\sigma}_j[f+g](S)\le \overline{\sigma}_j[f+g](S)\le\overline{\sigma}_jf(S)+\overline{\sigma}_jg(S).
\]
For (b), if $c=0$ or $c=-1$ the result is clear. For $c>0$ it follows from item $d)$ of Proposition \ref{requiredProperties}, from where, if $c<0$
\[\overline{\sigma}_j(cf)(S)= \overline{\sigma}_j(-c(-f))(S)= -c\overline{\sigma}_j(-f)(S)= c\overline{\sigma}_jf(S).\]
\end{proof}

\section{General Hypothesis to Establish $(L)$-a.e. and  proof of Theorem \ref{thm:L}} \label{app:L}

In  this appendix, we state and prove a generalization of Theorem \ref{thm:L}. It relies on the following concept of `good nodes'.

\begin{definition}[{\bf Good Nodes}] \label{admissibleNodes}
	A node $(S,j)$ is called \emph{good}, if $\mathcal{S}_{(S,j)}\neq N(S,j)$, (as introduced in \eqref{someNullSets}). Otherwise, $(S,j)$ is said to be \emph{bad}.
\end{definition}

The following theorem characterizes, under suitable assumptions, the good nodes as precisely those nodes at which the continuity from below property holds. Note that (H.5) relaxes the assumption of trajectorial completeness.

\begin{theorem}\label{proofOfL}
Suppose the trajectory set satisfies the following assumptions:
\begin{itemize}
	\item[(H.4)] If $(S,j)$ is a good up-down node, then for every $\epsilon>0$ there are $S^{\epsilon,1}, S^{\epsilon,2}\in \Se_{(S,j)}$ such that $(S^{\epsilon,1},j+1)$ and $(S^{\epsilon,2},j+1)$ are good nodes satisfying
	$$
	S^{\epsilon,1}_{j+1}-S_{j}\geq -\epsilon,\quad S^{\epsilon,2}_{j+1}-S_{j}\leq \epsilon
	$$
	\item[(H.5)] If  $(S^n)_{n\ge 0}$  is a sequence in $\mathcal{S}$ satisfying
	\begin{equation*}   
		~~~S^n_i = S^{n+1}_i \;\; 0 \leq i \leq n, ~~\forall ~n,~~~~~~~~
	\end{equation*}
 and if there is an $n_0\geq 0$ such that $(S^n,n)$ is a good node for every  $n\geq n_0$, then $ (S^n_n)_{n\in \mathbb{N}_0}\in \mathcal{S}$ (i.e. $\overline{S} \equiv \lim_{n \rightarrow \infty} S^n~\in \mathcal{S}$).
\end{itemize}
Then, $(L_{(S,j)})$ holds, if and only if $(S,j)$ is a good node.
\end{theorem}
\begin{proof}
Fix a good node $(S,j)$. Let $~~~~f = \Pi^{V^0, H^0}_{j, n_f}\in \mathcal{E}_{(S, j)}$
	and $f_m = \liminf_{n \rightarrow \infty} \Pi^{V^m, H^m}_{j, n}$ with $\Pi^{V^m, H^m}_{j, n} \in \mathcal{E}_{(S, j)}^+~\mbox{for all}~ n \geq j~\mbox{and}~ m \geq 1$ such that
	$$
	f \leq \sum_{m \geq 1} f_m\;\;\mbox{on}\; \Se_{(S,j)}.
	$$
	We need to show that 
	$$
	V^0 \leq \sum_{m \geq 1}  V^m=:V,
	$$
	(and, hence, can and will assume that $V<\infty$). Recall that
	$$
	N(S,j)=\{\tilde S\in \Se_{(S,j)}|\; (\tilde S,k) \mbox{ is an arbitrage node and }\tilde S_{k+1}\neq \tilde S_k \mbox{ for some } k\geq j\}.
	$$
	If $\tilde S\in \Se_{(S,j)} \setminus N(S,j)$, then, for every $k\geq j$, $(\tilde S,k)$ is an up-down node or $\tilde S_{k+1}=\tilde S_k$. Hence, by the Aggregation Lemma \ref{convergenceOfPortfolioCoordinates}, for every $\tilde S\in \Se_{(S,j)} \setminus N(S,j)$ and $n\geq j$
	$$
	\sum_{m\geq 1}  \Pi^{V^m, H^m}_{j, n}(\tilde S)=V+\sum_{i=j}^{n-1} H_i(\tilde S)\Delta_i\tilde S,
	$$
	for the non-anticipative sequence
	$$
	H_i(\hat S)=\begin{cases} \sum_{m=1}^\infty H^m_i(\hat S), & \mbox{if convergent in }\mathbb{R}, \\ 0,& \mbox{otherwise,} \end{cases} \quad i\geq j,\; \hat S\in \Se_{(S,j)}.
	$$
	Then, by Fatou's lemma, for every $\hat S\in \Se_{(S,j)}$
	$$
	f\leq V+\liminf_{n\rightarrow \infty} \sum_{i=j}^{n-1} H_i(\hat S)\Delta_i\hat S+\infty {\bf 1}_{ N(S,j)}(\hat S),
	$$
	which we may rearrange into
	$$
	V^0\leq V+ \liminf_{n\rightarrow \infty} \sum_{i=j}^{n-1} \tilde H_i(\hat  S)\Delta_i\hat S+\infty {\bf 1}_{ N(S,j)}(\hat S)
	$$
	where $\tilde H_i=H_i-H^0_i$ for $i<n_f$ and $\tilde H_i=H_i$ otherwise. 
	
	Thus, it is enough to show the following: For every $\delta>0$ there is an $\tilde S\in \Se_{(S,j)} \setminus N(S,j)$ such that
	\begin{equation}\label{eq:0007}
		\liminf_{n\rightarrow \infty} \sum_{i=j}^{n-1} \tilde H_i(\tilde  S)\Delta_i\tilde S\leq \delta.
	\end{equation}
To this end,  we construct a sequence $(S^n)$ in $\Se$ as follows: $S^n=S$ for $n\leq j$ and, inductively for $n> j$ in the following way, which guarantees that $(S^n,n)$ is a good node for every $n\geq j$. Assume $S^n$ has already been constructed for some $n\geq j$ and $(S^n,n)$ is a good node.

If $(S^n,n)$ is a flat node or an arbitrage node of type I, then choose $S^{n+1}\in \mathcal{S}_{(S^n,n)}$ such that $S^{n+1}_{n+1}=S^n_n$. We argue that $(S^{n+1},n+1)$ is good in the type I case (the flat case being similar and easier). Suppose to the contrary that $(S^{n+1},n+1)$ is bad. Then, for every $\hat S \in \mathcal{S}_{(S^{n+1},n+1)}\subseteq \mathcal{S}_{(S^{n},n)}$ there is an $i\geq n+1$ such that $(\hat S,i)$ is an arbitrage node and $\hat S_{i+1}\neq \hat S_i$ -- hence, $\hat S \in N(S^n,n)$. Moreover, any $\hat S \in \mathcal{S}_{(S^{n},n)} \setminus \mathcal{S}_{(S^{n+1},n+1)}$ belongs to $N(S^n,n)$ because $(S^n,n)$ is a type I arbitrage node. Thus, $\mathcal{S}_{(S^{n},n)}=N(S,n)$ -- a contradiction.

 If  $(S^n,n)$ is an up-down node, then, by (H.4), one can choose a sufficiently small $\epsilon>0$ and $S^{n+1}\in \mathcal{S}_{(S^n,n)} $   such that 
	\begin{equation}\label{eq:0008}
		\tilde H_n(S^n)(S_{n+1}^{n+1}-S^n_n)\leq 	|\tilde H_n(S^n)|\epsilon \leq  \delta 2^{-(n+1)}
	\end{equation}
and $(S^{n+1},n+1)$ is a good node.

Note that $(S^n,n)$ cannot be an arbitrage node of type II, because it is good by the inductive hypothesis. Hence, the construction of $S^{n+1}$ is finished.

	By (H.5),  $\tilde S=(S^n_n)_{n\in \mathbb{N}_0}\in \Se$. Then, by construction, $\tilde S\in\Se_{(S,j)}$ and $(\tilde S,n)=(S^n,n)$ is an up-down node (and then \eqref{eq:0008} holds) or $\tilde S_{n+1}=S^{n+1}_{n+1}=S^{n}_n=\tilde S_n$, whenever $n\geq j$. Hence, by construction, 
	$\tilde S\notin N{(S,j)}$ and 
	$$
	\liminf_{n\rightarrow \infty} \sum_{i=j}^{n-1} \tilde H_i(\tilde  S)\Delta_i\tilde S\leq \delta \sum_{i=j}^\infty 2^{-(i+1)} \leq \delta,
	$$
	which establishes \eqref{eq:0007}. Consequently, $(L_{(S,j)})$ holds.
	
	Conversely, assume that $(S,j)$ is a bad node. Then, by Lemma \ref{nullityOfN_j}, $\overline I_j1(S)=0$, which in turn implies $\overline \sigma_j0(S)\leq -1$. In view of Proposition \ref{Properties_L}, $(L_{(S,j)})$ does not hold.
\end{proof}

As an immediate consequence of Theorem \ref{proofOfL} and Lemma \ref{nullityOfN_j}, we obtain the following result:

\begin{corollary}\label{cor:L-a.e.}
	Under the assumptions of Theorem \ref{proofOfL}, additionally suppose that $(S,0)$ is good. Then $(L)$-a.e. holds.
\end{corollary}

We complete this appendix with the proof of Theorem \ref{thm:L}.

\begin{proof}[{\bf Proof of Theorem \ref{thm:L}}]
We first show that, under the assumptions of Theorem \ref{thm:L}, $(S,j)$ is bad, if and only if $(S,j)$ is an arbitrage node of type II. Since arbitrage nodes of type II are always bad, we only need to show that any other nodes are good. To this end, we fix a node $(S,j)$ which is not arbitrage of type II and construct a trajectory $\tilde S \in \mathcal{S}_{(S,j)}\setminus N(S,j)$. We set  $S^n=S$ for $n\leq j$ and inductively construct $S^n$ for $j\geq n$ in the following way, which guarantees that $(S^n,n)$ is not a type II arbitrage node: If $(S^n,n)$ is flat or an arbitrage node of type I, we choose $S^{n+1}\in \mathcal{S}_{(S^n,n)}$ such that $S^{n+1}_{n+1}=S^n_n$ and note that $(S^{n+1},n+1)$ is not an arbitrage node of type II by (H.2). If $(S^n,n)$ is an up-down node, then by (H.2) again, we find some $S^{n+1}\in  \mathcal{S}_{(S^n,n)}$ such that $(S^{n+1},n+1)$ is not an arbitrage node of type II. Since the trajectory set is complete, $\tilde S= (S^n_n)_{n\in \mathbb{N}_0}\in \mathcal{S}$. By construction, $\tilde S \in \mathcal{S}_{(S,j)}\setminus N(S,j)$.

We next argue that (H.2) implies (H.4). By the characterization of good nodes, (H.4) can be reformulated as
\begin{itemize}
	\item[(H.4')] If $(S,j)$ is an up-down node, then for every $\epsilon>0$ there are $S^{\epsilon,1}, S^{\epsilon,2}\in \Se_{(S,j)}$ such that 
	$$
	S^{\epsilon,1}_{j+1}-S_{j}\geq -\epsilon,\quad S^{\epsilon,2}_{j+1}-S_{j}\leq \epsilon
	$$
	and $(S^{\epsilon,1},j+1)$ and $(S^{\epsilon,2},j+1)$ are not arbitrage nodes of type II.
\end{itemize}
  To verify (H.4'), consider any up-down node $(S,j)$. If $(\hat S,j+1)$ is not a type II arbitrage node for every $\hat S\in \mathcal{S}_{(S,j)}$, then we find $S^1, S^2 \in \mathcal{S}_{(S,j)}$ such that $S^1_{j+1}-S_j>0$ and $S^2_{j+1}-S_j<0$. We may choose $S^{\epsilon,1}=S^{1}$ and  $S^{\epsilon,2}=S^{2}$ for every $\epsilon>0$. Otherwise, (H.2) directly applies.
  
  Finally, note that trajectorial completeness implies (H.5). Hence, by Theorem \ref{proofOfL}, $(L_{(S,j)})$ holds, if and only if $(S,j)$ is not an arbitrage node of type II. By (H.2), the initial node $(S,0)$ is not an arbitrage node of type II, and, thus, $(L)$ holds. Noting that the trajectories passing through type II arbitrage nodes form a null set (by Lemma \ref{nullityOfN_j}), we conclude $(L)$-a.e.  
	\end{proof}

\end{document}